\documentclass[12pt]{amsart}
\usepackage{amsmath,amscd,graphics,graphicx,color,a4wide,hyperref,verbatim}

\usepackage{multicol, fullpage}
\usepackage[usenames,dvipsnames]{xcolor} 
\usepackage{tikz, tikz-3dplot, pgfplots}
\usepackage{tkz-graph}
\usepackage{tikz-cd}
\usetikzlibrary[positioning,patterns] 
\usetikzlibrary{matrix,arrows,decorations.pathmorphing}

\usepackage{graphicx}
\usepackage{psfrag}
\usepackage{marvosym}
\usepackage{amsmath}
\usepackage{amsfonts}
\usepackage{amssymb}
\usepackage{amsthm}
\usepackage{mathrsfs}
\usepackage{enumitem}

\usepackage{ytableau}

\usepackage{comment}

\usepackage{calligra}
\usepackage{mathrsfs}
\DeclareMathAlphabet{\mathcalligra}{T1}{calligra}{m}{n}
\DeclareFontShape{T1}{calligra}{m}{n}{<->s*[1.5]callig15}{}

\newtheorem{theorem}{Theorem}[section]
\newtheorem{lemma}[theorem]{Lemma}
\newtheorem{lem-def}[theorem]{Lemma-definition}
\newtheorem{proposition}[theorem]{Proposition}
\newtheorem{prop-def}[theorem]{Proposition-definition}

\newtheorem{corollary}[theorem]{Corollary}

\theoremstyle{definition}
\newtheorem{definition}[theorem]{Definition}

\newtheorem{example}[theorem]{Example}

\newtheorem{remark}[theorem]{Remark}
\newtheorem*{questionstar}{Question}

\newtheorem*{theoremstar}{Theorem}

\numberwithin{equation}{section}

\newtheorem{thm}{Theorem}[section] 
\theoremstyle{plain} 
\newcommand{\thistheoremname}{}
\newtheorem{genericthm}[thm]{\thistheoremname}
  
\newtheorem*{genericthm*}{\thistheoremname}
\newenvironment{namedthm*}[1]
  {\renewcommand{\thistheoremname}{#1}%
   \begin{genericthm*}}
  {\end{genericthm*}}

\renewcommand{\AA} {\mathbb{A}}

\newcommand{\CC} {\mathbb{C}}

\newcommand{\LL} {\mathbb{L}}

\newcommand{\NN} {\mathbb{N}}

\newcommand{\PP} {\mathbb{P}}

\newcommand{\RR} {\mathbb{R}}

\newcommand{\ZZ} {\mathbb{Z}}

\newcommand {\shA} {\mathcal{A}}
\newcommand {\shB} {\mathcal{B}}
\newcommand {\shC} {\mathcal{C}}
\newcommand {\shD} {\mathcal{D}}

\newcommand {\shF} {\mathcal{F}}

\newcommand {\shH} {\mathcal{H}}

\newcommand {\shR} {\mathcal{R}}

\newcommand {\shT} {\mathcal{T}}
\newcommand {\shP} {\mathcal{P}}

\newcommand {\sE} {\mathscr{E}}

\newcommand {\sL} {\mathscr{L}}
\newcommand {\sM} {\mathscr{M}}

\newcommand {\sO} {\mathscr{O}}

\newcommand {\foa}  {\mathfrak{a}}

\newcommand {\coh} {\operatorname{coh}}

\newcommand {\cone} {\operatorname{cone}}

\newcommand{\sExt}{\mathscr{E} \kern -3pt xt}

\newcommand {\Gr} {\operatorname{Gr}}

\newcommand {\Hom} {\operatorname{Hom}}
\newcommand {\sHom}{\mathscr{H}\kern-5pt\mathcalligra{om}}

\newcommand {\id} {\operatorname{id}}
\newcommand {\Id} {\operatorname{Id}}

\newcommand {\kk} {\Bbbk}

\newcommand {\Ker} {\operatorname{Ker}}

\newcommand {\Pf} {\operatorname{Pf}}

\newcommand {\rank} {\operatorname{rank}}

\newcommand {\Spec} {\operatorname{Spec}}

\newcommand {\Sym} {\operatorname{Sym}}

\newcommand {\Tor} {\operatorname{Tor}}

\newcommand{\hpd}{{\natural}}

\newcommand{\Perf}{\mathrm{Perf}}
\newcommand{\perf}{\mathrm{perf}}
\newcommand{\Bl}{\mathrm{Bl}}
\newcommand{\act}{\operatorname{act}}
\newcommand{\rBl}{\mathrm{Bl}^{\rm ref}}
\newcommand{\rBlA}{\widetilde{\shA}^{\rm ref}}

\DeclareRobustCommand{\Sec}{\ifmmode\mathsection\else\textsection\fi}

\title[]{Blowing up linear categories, refinements, and homological projective duality with base locus}
\author[Q.Y.\ JIANG, N.C.\ Leung]{Qingyuan Jiang, Naichung Conan Leung}

\address{Institute for Advanced Study, Einstein Drive, Princeton, NJ 08540, USA}\email{jiangqy@ias.edu}

\address{The Institute of Mathematical Sciences and Department of Mathematics,
The Chinese University of Hong Kong, Shatin, N.T., Hong Kong}\email{leung@math.cuhk.edu.hk}

\begin{document}

\begin{abstract}
\noindent
In this paper, we first introduce geometric operations for linear categories, and as a consequence generalize Orlov's blow up formula \cite{Orlov04} to possibly singular local complete intersection centres. Second, we introduce {\em refined} blowing up of linear category along base--locus, and show that this operation is dual to taking linear intersections. Finally, as an application we produce examples of Calabi--Yau manifolds which admits Calabi--Yau categories fibrations over projective spaces.

\end{abstract}

\maketitle


\section{Introduction}

Homological projective duality (HPD) introduced by Kuznetsov \cite{Kuz07HPD}, has been a very fruitful theory to produce interesting semiorthogonal decompositions of derived categories of algebraic varieties, and also to relate derived categories of different varieties, see \cite{RT15HPD} and \cite{Kuz14SOD} for nice surveys, or \cite{JLX17} for a review and references therein. An important question is how to produce examples of HPD, and one useful strategy would be to produce new HPDs from existing ones. In this paper we will focus on the question:

\begin{questionstar} What is the HPD of the linear sections of a given HPD pair?
\end{questionstar}

More precisely, assume $V$ and $V^\vee$ are dual vector spaces (or more generally, dual vector bundles), and suppose $X \to \PP(V)$ and $Y \to \PP(V^\vee)$ are HPD Lefschetz varieties (or more generally, Lefschetz categories) of length $m$ and $n$, and $L \subset V^\vee$ is a generic linear subspace of dimension $\ell$. 
Then the goal is to find the HPD of the linear section $Y_L  = Y \times_{\PP(V^\vee)} \PP(L)$.

The question is answered by Carocci-Turcinovic \cite{CT15} in the case when the base locus $X_{L^\perp} \subset X$ of the linear system $L$ is smooth and is of large codimension. More precisely, if the codimension $\ell$ of $X_{L^\perp} \subset X$ satisfies $\ell > m$, then they showed the HPD of $Y_L$ is given by blowing up $\Bl_{X_{L^\perp}} X$ of $X$ along the base locus $X_{L^\perp}$. The problem for this result to hold in general is that, if the codimension is not large enough (i.e. $\ell \le m$), then the category of $\Bl_{X_{L^\perp}} X$ would be too large in general to be the proper HPD of $Y_L$.

In this paper we resolve this problem by introducing the notion of {\em refined categorical blowing up} $\Bl^{\rm ref}_{X_{L^\perp}} X$, and show that the HPD of $Y_L$ is always given by $\Bl^{\rm ref}_{X_{L^\perp}} X$.  This completes the answer to the above question,  generalizes the result of \cite{CT15} to any codimension $\ell$ and to $\PP(V)$-linear categories, and drops the smoothness assumption on the base locus $X_{L^\perp}$.

\subsection{Geometric operations on linear categories} First we consider the general question of how to perform geometric operations on linear categories. We start with the question: if $\shA \subset D(X) := D^b_{\coh} (X)$ is an admissible subcategory of a regular scheme $X$ over a field of characteristic zero, $Y \subset X$ is a  regular subscheme of codimension $r \ge 2$, then what is the blowing up of the category $\shA$ along $Y$?

The expected answer should be compatible with the commutative counterpart: namely if we assume $\shA$ is given by a scheme $X'$ with $X' \to X$, then the blowing up category $\widetilde{\shA}$ should be given by the category of the blowing up scheme $\widetilde{X'}$ of $X'$ along $Y': = X ' \times_X Y$. From blowing up closure formula for schemes, the blowing up $\widetilde{X'}$ can be obtained as the proper transform of $X'$ along the blowing up $\Bl_{Y} X \to X$, namely the scheme-theoretic closure of inverse image of $X' \backslash Y'$ along  $X'\times_X \Bl_Y X \to X'$. However, the operation of strict transform of schemes seems to be very difficult in general to be lifted to categorical level.

Fortunately, the question can be answered in the situation when $\shA$ is a $S$-linear subcategory, using the knowledge of linear categories developed by Kuznetsov \cite{Kuz11Bas}. 
This situation includes the case of blowing up of projective varieties along base locus of a linear system $L$ (where $S = \PP(V)$), which would be the main case for our later applications.

More precisely, suppose $S$ is a regular scheme, $Z \subset S$ is a regular closed subscheme of codimension $r \ge 2$, and $X$ is a $S$-scheme, $Y = Z \times_S X \subset X$ is of expected codimension $r$. Then $D(X)$ admits an action of the category $D(S)$ under (derived) pulling back along $X \to S$, and the subcategory $\shA \subset D(X)$ is called $S$-linear if it is closed under the action of $D(S)$ (see \S \ref{sec:bc}). Then the {\em blowing up $\widetilde{\shA}$ of the $S$-linear category $\shA \subset D(X)$} along $Y \subset X$ (or more accurately, along $\shA_Y$) can be defined (Def. \ref{def:blcat}) using the techniques of base-change of the $S$-linear $\shA$ developed by Kuznetsov \cite{Kuz11Bas} along a morphism $\widetilde{S} = \Bl_Z S \to S$.  We show that Orlov's results on blowing up hold also for the blowing up of category $\widetilde{\shA}$ of $\shA$, see Thm. \ref{thm:bl} for the precise statement, and \S \ref{sec:bl} for more detail.  

Notice that even in the commutative case,  our approach gives something new, namely a blowing-up formula for possibly singular centers: suppose $Y$ is a codimension $r \ge 2$ local complete intersection (l.c.i.) subscheme of a smooth scheme $X$ over a filed of characteristic zero, then Orlov's blowing up formula holds for the blowing up of $X$ along $Y$ without smoothness assumption on $Y$, see Cor. \ref{cor:bl:lci}. Notice that the smoothness of $X$ is not even necessary, as long as it can be realized as fiber products of smooth ones from Tor-independent squares, see Rmk. \ref{rmk:thm:bl:nreg} for more precise statement.

The above approach can also be carried to other geometric operations: projective bundle, (generalized) universal hyperplane, etc and we show the corresponding formulae on categories, see \S \ref{sec:app:gem:cat} for more details.

\subsection{Refined blowing up and HPD with base locus}
Let $X \to \PP(V)$ be a smooth projective variety. The input data for HPD theory is a {\em Lefschetz category}, namely an $\PP(V)$-linear admissible subcategory $\shA \subset D(X)$ with a {\em Lefschetz structure}, i.e. $\shA$ admits a semiorthogonal decomposition of the form
	$$\shA = \langle \shA_0, \shA_1(1), \ldots, \shA_{m-1}(m-1)\rangle,$$
where $\shA_0 \supset \shA_1 \supset \ldots \supset \shA_{m-1} \ne 0$ are admissible subcategories of $\shA$, and $\shA_*(k)$ denotes the image of the category $\shA_*$ under the autoequivalence $-\otimes \sO_{\PP(V)}(k) \colon \shA \to \shA$, $k \in \ZZ$. The number $m \in \NN$ is called the {\em length} of the Lefschetz category $\shA$, sometimes denoted by ${\rm length} (\shA)$. In the original HPD theory \cite{Kuz07HPD} it is required that $\shA$ is {\em moderate}, i.e. ${\rm length} (\shA) < \rank V$, and all applications of HPD hold under this condition. We will also stick to this convention and requires all Lefschetz categories to be moderate. Moreover, we show that for any {non-moderate} Lefschetz category $\shA$ (i.e. ${\rm length} (\shA) \ge \rank V$), it can always be {refined} to be an honest (=moderate) Lefschetz category, see Lem \ref{lem:lef:nonmoderate}.

The {\em HPD category $\shA^\hpd$} of the $\PP(V)$-linear Lefschetz category $\shA$, is itself a $\PP(V^\vee)$-linear Lefschetz category, and by definition captures the essential categorical information of the category $\shA$ as a family over $\PP(V)$. See Def. \ref{def:HPDcat} for the precise definition. In commutative case $\shA = D(X)$ and $\shA^\hpd= D(Y)$ for a variety $Y \to \PP(V^\vee)$, then $Y$ is the homological modification of the classical projective dual variety $X^\vee \subset \PP(V^\vee)$ of $X$ and we can recover $X^\vee$ as its critical values. The HPD is a duality relation: namely $\shA$ is also the HPD of $\shA^\hpd$. See \cite{Kuz07HPD, RT15HPD, JLX17} for more about HPD.

Back to the situation of the presence of a linear system $L \subset V^\vee$ of dimension $\ell$. Then the restriction $\shA^\hpd_{\PP(L)}$ of $\shA^\hpd$ along the inclusion $\PP(L) \subset \PP(V^\vee)$ is a $\PP(L)$-linear category. However $X \to \PP(V) \dashrightarrow \PP(L^\vee)$ is only a rational morphism, and $X_{L^\perp}= X \times_{\PP(V)}  \PP(L^\perp)$ is the base locus of $X$, which we suppose to also have codimension $\ell$, where $L^\perp = \Ker\{V \to L^\vee\} \subset V$ is the orthogonal linear subspace of $L$. Then the construction of previous section allows us to blow up $\shA \subset D(X)$ along base locus, to obtain a $\PP(L^\vee)$-linear category $\Bl_{\shA_{\PP(L^\perp)}} \shA$. The pair
	$$\text{$\PP(L^\vee)$-linear category $\Bl_{\shA_{\PP(L^\perp)}} \shA$} \quad \text{and} \quad \text{$\PP(L)$-linear category $\shA^\hpd_{\PP(L)}$} $$
then has a chance to be HPD over the dual projective spaces $\PP(L^\vee)$ and $\PP(L)$. The problem is that in general the category $\Bl_{\shA_{\PP(L^\perp)}} \shA$ is too large to be a Lefschetz category. 

The problem can be solved as follows. We observe that (see Lem. \ref{lem:app:ref-bl}) the $\PP(L^\vee)$-linear category $\Bl_{\shA_{\PP(L^\perp)}} \shA$ contains ``redundant'' components 
	$$
\big \langle \shA_{\ell-1} \boxtimes D(\PP(L^\vee)), \ldots, \shA_{m-1}(m-\ell) \boxtimes D(\PP(L^\vee)) \big \rangle  \subset \Bl_{\shA_{\PP(L^\perp)}} \shA
	 $$
which can be regarded as ``trivial family" over $\PP(L^\vee)$. Therefore their right orthogonal, which is defined to be the {\em refined blowing up $ \Bl_{\shA_{\PP(L^\perp)}}^{\rm ref} \shA$ of $\shA$ along $\shA_{\PP(L^\perp)}$}, contains the essential information of $\Bl_{\shA_{\PP(L^\perp)}} \shA$ as a $\PP(L^\vee)$-linear category. Our main result is the following: 

\begin{theoremstar}[Thm. \ref{thm:app:HPDbs}] The pair of (moderate) Lefschetz categories:
$$\text{$\PP(L^\vee)$-linear category $\Bl_{\shA_{\PP(L^\perp)}}^{\rm ref} \shA$} \quad \text{and} \quad \text{$\PP(L)$-linear category $\shA^\hpd_{\PP(L)}$} $$
are homological projective dual (HPD) to each other.
\end{theoremstar}

If $\shA = D(X)$ and $\ell > m$, then $\Bl_{\shA_{\PP(L^\perp)}}^{\rm ref} \shA = \Bl_{\shA_{\PP(L^\perp)}} \shA = D(\Bl_{X_{L^\perp}}X)$,
the theorem reduces to the result of \cite{CT15} without smoothness assumption on $X_{L^\perp}$. In general if $\ell \le m$, then $\Bl_{\shA_{\PP(L^\perp)}}^{\rm ref} \shA$ will be a strictly smaller subcategory of the usual blowing up.

As the case of all considerations of HPD theory, it is expected that the expected dimension condition of $X_{L^\perp} \subset X$ can be dropped if one consider {\em derived} fiber product $X \times_{\PP(V)} \PP(L^\perp)$ instead of scheme-theoretic fiber product. Our constructions and results in this paper should shed lights on the question of what the {\em derived blowing up} should be, at least in the case of blowing up along base-locus of a linear system.

\subsection{Application to fibrations of Calabi--Yau categories}
As an application of our construction and theorem, consider for integers $d, k \ge 2$, the Calabi--Yau $(kd-k-1)$-fold $X_{(d,d, \ldots, d)} \subset \PP^{kd-1}$ which is a complete intersection of $k$ general degree $k$ hypersurfaces inside $\PP^{kd-1}$ (i.e. of type $(d,d, \ldots, d)$). The usual blowing up of $\PP^{kd-1}$ along $X_{(d,d, \ldots, d)}$ consist of redundant information of $\PP^{kd-1}$; However the {\em refined} blowing up with respect to the action $\otimes \sO_{\PP^{kd-1}}(d)$ is essentially  ``$\PP^{k-2}$-copies" of $X_{(d,d, \ldots, d)}$. Then our theorem implies that 
	$$D(X_{(d,d,\ldots, d)}) \simeq \shA^{\hpd}_{\PP^{k-1}}$$
where $\shA^{\hpd}_{\PP^{k-1}}$ is a family of Calabi--Yau categories of dimension $k(d-2)$ over $\PP^{k-1}$, with fiber $\shA^{\hpd}_{s}$ over a general point $s \in \PP^{k-1}$ given by the Kuznetsov component  \cite{KuzCY} of the derived category of a degree $d$ hypersurface $X_s\subset \PP^{kd-1}$:
	$$D(X_s) = \big \langle \shA^{\hpd}_{s}, ~ \sO, \sO(1), \ldots, \sO(kd-d-1)\big \rangle.$$

For example, if we take $(k,d) = (3,2)$, we obtain the Calabrese--Thomas derived equivalence between a Calabi--Yau threefold $X_{(3,3)} \subset \PP^4$ and a pencil of K3 categories from cubic fourfolds \cite{CT16}.  Similarly, we also obtain that the Calabi--Yau $5$-fold $X_{(4,4)} \subset \PP^{7}$ admits a pencil of Calabi--Yau categories of dimension $4$; The Calabi--Yau $5$-fold $X_{(3,3,3)}\subset \PP^{8}$ admits a Calabi--Yau $3$ category fibration over $\PP^2$, etc. See \S \ref{sec:CY} for more details.

\subsection{Application to homological projective geometry}
The work in this paper fits into the framework of {\em homological projective geometry} (see \cite{JLX17, KP18, JL18join} for more details) as follows. Denote  by $\underline{{\rm Lef}}_{/\PP(V)}$ the category of smooth proper $\PP(V)$-linear Lefschetz categories, and fix any linear system (i.e. linear subbundle) $L \subset V^\vee$. Then our main result shows that there are two functors of Lefschetz categories: for the linear inclusion $\PP(L) \subset \PP(V^\vee)$, there is a restriction functor:
		$$  (-)|_{\PP(L)}   \colon \underline{{\rm Lef}}_{/\PP(V^\vee)}  \to \underline{{\rm Lef}}_{/\PP(L)}, \qquad \shA \mapsto \shA|_{\PP(L)};$$
	and for the linear projection $\PP(V) \dashrightarrow \PP(L^\vee)$, there is a refined blowing up functor:
		$$ \Phi(-): = \Bl_{\PP(L^\perp)}^{\rm ref} (-)  \colon \underline{{\rm Lef}}_{/\PP(V)}  \to \underline{{\rm Lef}}_{/\PP(L^\vee)}, \qquad \shA \mapsto \Phi(\shA) :=\Bl_{\shA_{\PP(L^\perp)}}^{\rm ref} \shA,$$
	(if we assume the smoothness of the intersections of $\shA_{\PP(L)}$ or respectively $\shA_{\PP(L^\perp)}$), and that these two operations are {\em dual} to each other under HPD, i.e. $\Phi(\shA)^\hpd = (\shA^\hpd)|_{\PP(L)}$, $\shA \in  \underline{{\rm Lef}}_{/\PP(V^)}$.	

The construction in this paper also allows us to define categorical joins in \cite{JL18join} for Lefschetz varieties with nontrivial intersections, and show the duality between categorical joins and fiber products. 

Notice there is another construction called categorical cones $\shC_{\PP(L^\vee)} (-)\colon \underline{{\rm Lef}}_{/\PP(L^\vee)}  \to \underline{{\rm Lef}}_{/\PP(V)}$, which is a special case of categorical joins if there is a splitting $V = L^\vee \oplus L^\perp$, see \cite{KP18, JL18join}. Then the combination of results of this paper and of \cite{JL18join} shows that the composition:
		$$\Phi  \circ \shC_{\PP(L^\vee)} (-) \colon   \underline{{\rm Lef}}_{/\PP(L^\vee)}  \to \underline{{\rm Lef}}_{/\PP(V)} \to \underline{{\rm Lef}}_{/\PP(L^\vee)} $$
is equivalent to identity, i.e. $\Phi (\shC_{\PP(L^\vee)} (\shA)) \simeq \shA$, for $\shA \in \underline{{\rm Lef}}_{/\PP(L^\vee)}$.

All these results provide further evidences for the theory of homological projective geometry and in turn shows the naturality of the construction of this paper.

\subsection{Convention.} Let $B$ be a fixed base scheme, smooth over a ground field of characteristic zero, and $V$, $V^\vee$ be dual vector bundles of rank $N$ over $B$. All schemes considered in this paper will be $B$-schemes, and products are fiber products over $B$. A $B$-linear category will be an {\em admissible} subcategory $\shA \subset D(X)$ for some $B$-scheme $X$. 

The constructions, arguments, and results of this paper -- at least the part on perfect complexes -- should have no difficulty to be carried out in $dg$-setting or $\infty$-setting of Lurie \cite{Lur} over any base scheme $B$. For example the noncommutative HPD theory has been set up in Lurie's framework by Alex Perry \cite{P18}. However in this paper we restrict ourselves to the above convention for following reasons: our results are mainly about the bounded coherent derived categories rather than perfect complexes; our approach depends on the geometry of blowing up and generalized hyperplane section, and the categorical formulae for them are so far only established for characteristic zero by Orlov, and trying to prove them in arbitrary characteristic will go beyond the aim of this paper; the overall well-established frameworks of Bondal, Orlov, Kuznetsov are set up for this convention at this stage. 

We use notation $X,Y,Z,S,T, \ldots$ to denote schemes, $\sL, \sE, \sM$ to denote coherent sheaves or vector bundles on certain schemes, and $\shA, \shB, \shC, \ldots$ to denote triangulated categories and $A,B,C \ldots$ to typically denote the elements of the corresponding categories. Functors considered in this paper are all {\em derived} unless specified otherwise.

\subsection*{Acknowledgement} J.Q. would like to thank Richard Thomas, Zak Turcinovic and Carocci Francesca for many helpful discussions on HPD, Cayley's trick, and the work \cite{CT15}. The authors are grateful for Mikhail Kapranov,  Andrei C\v{a}ld\v{a}raru, Matthew Young and Ying Xie for many helpful discussions. J.Q would also like to thank Yu-Wei Fan for bringing his attention to fibrations of Calabi--Yau categories. J.Q. is supported by Grant from National Science Foundation (Grant No. DMS -- 1638352) and the Shiing-Shen Chern Membership Fund of IAS; L.N.C. is supported by grant from the Research Grants Council of the Hong Kong Special Administrative Region, China (Project No. CUHK -- 14301117).

\section{Preliminaries}

\subsection{Generalities} A {\em semiorthogonal decomposition} of a triangulated category $\shT$,
	$$\shT = \langle \shA_1, \ldots, \shA_n \rangle,$$
is a sequence of admissible full triangulated subcategories $\shA_1, \shA_2, \ldots, \shA_{n}$, such that $(i)$ $\Hom (a_j ,a_i) = 0$ for all $a_i \in \shA_i$ and $a_j \in \shA_j$ , if $j > i$, and $(ii)$ they generate the whole $D(X)$. 
A full triangulated subcategory $\shA$ of (a triangulated category) $\shT$ is called {\em admissible} if the inclusion functor $i = i_{\shA}: \shA \hookrightarrow \shT$ has both a right adjoint functor $i^!: \shT \to \shA$ and a left adjoint functor $i^*: \shT \to \shA$. If $\shA\subset \shT$ is admissible, then its {\em right orthogonal} $\shA^\perp = \{ T \in \shT \mid \Hom(\shA,T) = 0\}$ and {\em left orthgonal} ${}^\perp \shA =\{ T \in \shT \mid \Hom(T, \shA) = 0\}$ are both admissible, and $\shT = \langle \shA^\perp, \shA \rangle =  \langle \shA, {}^\perp \shA\rangle$. The functors $\LL_\shA: = i_{\shA^\perp} i^*_{\shA^{\perp}}$  (resp. $\RR_{\shA} : =  i_{{}^\perp \shA} i^!_{ {}^\perp \shA}$) is called the {\em left (resp. right) mutation passing through $\shA$}. The mutation functors allow us to start with a semiorthogonal decomposition to obtain a whole sequence of new semiorthogonal decompositions. The readers are referred to \cite{Huy} and \cite{BO} for definitions and properties of derived categories of algebraic varieties and semiorthogonal decompositions, and to \cite{B, BK} or reviews in \cite{Kuz07HPD, JLX17} for more about mutation functors. 

\subsection{Derive categories over a base.} \label{sec:bc} Let $S$ be a fixed scheme. A $S$-linear category is a (dg-enriched or stable $\infty$-enriched) triangulated category equipped with a module structure over the commutative algebra object $\Perf(S)$. If $X$ is a $S$-scheme with structure map $a_X\colon X \to S$, then $\Perf(X)$, $D(X)$ and $D_{qc}(X)$ are naturally equipped with $S$-linear structure by $A \otimes a_X^* F$, for any $F \in \Perf(S)$ and $A \in \Perf(X)$, $D(X)$ or $D_{qc}(X)$. We will mainly focus on admissible subcategories $\shA \subset D(X)$ of the bounded coherent derived category $D(X)$. For simplicity we will only consider quasi-compact schemes proper over $S$.

\begin{definition}[{\cite{Kuz07HPD,Kuz11Bas}}] 
\begin{enumerate}[leftmargin=*]
\item An admissible subcategory $\shA \subset D(X)$ for a $S$-scheme $a_X \colon X\to S$ is called {\em $S$-linear} if $A \otimes a_X^* F \in \shA$ for any $A \in \shA$ and $F \in \Perf(S)$. 
\item For $S$-linear categories $\shA \subset D(X)$ and $\shB \subset D(Y)$, where $X,Y$ are $S$-schemes with structural morphisms $a_X$ and $a_Y$, an exact functor $\Phi\colon \shA \to \shB$ is called {\em $S$-linear} if there are functorial isomorphisms $\Phi(A \otimes a_X^* F) \simeq \Phi(A) \otimes a_Y^* F$ for any $A \in \shA$ and $F \in \Perf(S)$.
\end{enumerate}
\end{definition}

An $S$-linear category $\shA$ is by definition equipped with an action functor
	$$\act \colon \shA \boxtimes D(\PP(L^\vee)) \to \shA,$$
given by $\act(A \boxtimes F) = A \otimes a_X^* F \in \shA$, where $\shA \boxtimes D(\PP(L^\vee))$ is to be defined later.

\begin{definition}[{\cite{Kuz07HPD,Kuz11Bas}}] \label{def:bc} A base change $\phi: T \to S$ is called \emph{faithful} for the morphism $a: X \to S$ (or {\em faithful} for the $S$-scheme $X/S$) if for Cartesian square
	\begin{equation}\label{eqn:fiber}
	\begin{tikzcd}
	X_T = X \times_S T  \arrow{d}{}[swap]{a_T}  \arrow{r}{\phi_T} & X \arrow{d}{a} \\
	T\arrow{r}{\phi}  & S
	\end{tikzcd}
	\end{equation}
the natural transformation $ a^* \circ \phi_* \to \phi_{T *} \circ a_T^*: D(T) \to D(X)$ is an isomorphism. 
\end{definition}
It is well-known that a base-change $\phi$ as above is faithful for $X/S$ if and only if the square (\ref{eqn:fiber}) is {\em Tor-independent}, i.e. for all $t \in T$, $x \in X$, and $s \in S$ with $\phi(t) = s = a(x)$, $\Tor_i^{\sO_{S,\,s}}(\sO_{T,\,t}, \sO_{X,\,x}) = 0$ for all $i > 0$.

\begin{definition}[{\cite{Kuz11Bas}}] \label{def:bc-A} Let $\shA$ be a $S$-linear admissible subcategory of $D(X)$. We assume $\phi: T \to S$ is a {faithful} base-change for $X/S$. Then the base-change category of perfect complexes category $\shA^\perf$ along $T \to S$ is defined to be:
	$$\shA_T^\perf = \shA^\perf \otimes_{\Perf(S)} \Perf(T) \subset \Perf(X_T) = \Perf(X) \otimes_{\Perf(S)} \Perf(T),$$
by which we mean the $\shA_T^{\rm perf} \subset \Perf(X_T)$ is the subcategory thickly generated by (i.e. the minimal idempotent-complete triangulated subcategory containing) elements $\phi_T^* A \otimes f_T^* F$, for all $A \in \shA^{\rm perf} : = \shA \cap \Perf(X)$ and $F \in \Perf(T)$. Let $\hat{\shA_T} \subset D_{qc}(X_T)$ be the minimal subcategory containing $\shA_T^{\rm perf}$ and closed under arbitrary direct sums. The {\em base-change of $\shA$ along $\phi$} is the $T$-linear admissible subcategory defined by
	$$\shA_T := \hat{\shA_T}  \cap D(X_T) \subset D(X_T).$$
The category $\shA_T$ satisfies: $\phi_T^*(a) \in \shA_{T}$ for any $a \in \shA$ and $\phi_{T*}(b) \in \shA$ for $b \in \shA_{T}$ with proper support over $X$, see {\cite{Kuz11Bas}}. Then the $T$-linear category $\shA_T$ depends only on $\shA$ and $\phi$, and does not depend on the choice of an embedding $\shA \subset D(X)$. 
\end{definition}

\begin{remark} The tensor product in the above definition $\shA_T^\perf = \shA^\perf \otimes_{\Perf(S)} \Perf(T)$ agrees with the tensor products of small idempotent-complete stable $\infty$-categories in the sense of \cite{Lur} and \cite{BFN}. More precisely, for two small idempotent-complete stable $\infty$-categories $\shC_1$ and $\shC_2$, then tensor product is defined to be $\shC_1 \otimes \shC_2 = ({\rm Ind}(\shC_1) \otimes {\rm Ind}(\shC_2))^c$, where ${\rm Ind}$ is ${\rm Ind}$-completion, and $(-)^c$ is taking compact objects, see \cite{Lur}, \cite{BFN} or \cite{P18}. If $X$ and $X_T$ are smooth, then $\shA_T = \shA_T^{\rm perf}$. In general they are related by:
	$\shA_T^\perf = \shA_T \cap \Perf(X_T)$ by \cite{Kuz11Bas}, and it is expected that the categorical Poincar\'e duality holds: $\shA_T = {\rm Fun}^{ex}_{\Perf(S)}(\shA_T^\perf, D(S))$ (for their stable $\infty$-enhancements), since this duality holds if $\shA = D(X)$ for perfect proper stacks $X \to S$ by  \cite{BNP}.
\end{remark}

\begin{lemma}[{\cite[Cor. 5.7]{Kuz11Bas}}] \label{lem:bc-char} In the situation of Def. \ref{def:bc-A} , we assume further the base-change $\phi$ is {\em projective} (therefore $\phi_{T *} (D(X_T)) \subset D(X)$, the pushforward preserves coherence). Then the $T$-linear admissible subcategory $\shA_T \subset D(X_T)$ is characterized by
	$$\shA_T  = \{ C \in D(X_T) \mid \phi_{T*} (C \otimes f^* F) \in \shA, \quad \forall F \in \Perf(T)\} \subset D(X_T).$$
\end{lemma}

\begin{remark} Suppose $\phi_T$ is a proper perfect morphism, then $\phi_{T *} $ preserves perfect complexes and we also have a characterization
	$$\shA_T^\perf  = \{ C \in \Perf(X_T) \mid \phi_{T*} (C \otimes f^* F) \in \shA^\perf, \quad \forall F \in \Perf(T)\} \subset \Perf(X_T).$$
\end{remark}

\begin{lemma}[Compatibility] Assume $\phi \colon T \to S$ and $\psi \colon T' \to T$ are faithful base-changes for $X/S$ and $X_T/T$ respectively (therefore the composition $\phi'  = \psi \circ \phi \colon T' \to T \to S$ is a faithful base-change for $X/S$), and $\shA \subset D(X)$ is an $S$-linear admissible subcategory. Then
	$$(\shA_T)_{T'} = \shA_{T'} \subset D(X_{T'}),$$
where the left hand side is the base-change of $\shA$ along $T\to S$ followed by the base-change along $T' \to T$, and the right hand side is the base-change of $\shA$ along composition $\phi' \colon T \to T' \to S$.
\end{lemma}
\begin{proof} By Def. \ref{def:bc-A}, notice the perf-version $(\shA_T)_{T'} ^\perf$ of left hand side is thickly generated by 		
	$$\psi^* ( \phi^* A \otimes F) \otimes F' = \phi'^* A \otimes (\psi^* F \otimes F') , \qquad A \in \shA, F \in \Perf(T), F' \in \Perf(T'),$$
and the perf-version $\shA_{T'}^\perf$ for right hand side is thickly generated by 
	$$ \phi'^* A  \otimes F' , \qquad A \in \shA, F' \in \Perf(T').$$
Since $\psi^* F \otimes F' \in \Perf(T')$, therefore $(\shA_T)_{T'} ^\perf$ and $  \shA_{T'} ^\perf$ are the same subcategory of $\Perf(X_{T'})$. Taking completion in $D_{qc}(X_{T'})$ under arbitrary direct sums, and then intersecting with $D(X_{T'})$, the lemma follows.
\end{proof}

\begin{proposition} [\cite{Kuz07HPD}, Prop. 2.39]\label{prop:bcFM} If $D^b(Y) = \langle \Phi_{K_1}(D^b(X_1)), \ldots, \Phi_{K_n}(D^b(X_n))\rangle$ is a semi-orthogonal decomposition, where $X_i$ and $Y$ are projective over $S$ and smooth, $i=1,\ldots, n$ and $K_i \in D^b(X_i \times_S Y)$ Fourier-Mukai kernels. Assume $\phi: T \to S$ is faithful for all the pairs $(X_1, Y),\ldots,(X_n,Y)$ (i.e. $\phi$ is faithful for $X_i/S$, $Y/S$ and $X_i \times_S Y/S$, $i=1,\ldots,n $). Then the functor $\Phi_{K_{iT}} \colon D(X_T) \to D(Y_S)$ is fully faithful for $i=1,\ldots,n$, and there is a $T$-linear semi-orthogonal decomposition 
	$$D^b(Y_T) = \langle \Phi_{K_{1T}}(D^b(X_{1T})), \ldots, \Phi_{K_{nT}}(D^b(X_{nT}) .\rangle$$
\end{proposition}

\begin{proposition} [{\cite[Thm. 5.6]{Kuz11Bas}}]\label{prop:bcsod} If $f:X \to S$ is a morphism, $D(X) = \langle \shA_1, \ldots, \shA_n \rangle$ is a $S$-linear semiorthogonal decomposition by admissible subcategories, such that the projection $D(X) \to \shA_k$ is of finite cohomological amplitude, for $k=1,\ldots,n$. Let $\phi: T \to S$ be a faithful base-change for $f$, then there is a $T$-linear semiorthogonal decomposition
	$$D(X_T) = \langle \shA_{1T}, \ldots, \shA_{nT} \rangle$$
where $\shA_{kT}$ is the base-change category of $\shA_k$ along $T \to S$.
\end{proposition}

The above two propositions also holds for perfect complexes, by the way how the base-change categories are constructed and these propositions are proved in {\cite[Thm. 5.6]{Kuz11Bas}}.

\begin{definition}[{\cite{Kuz11Bas}}]  Assume $\shA \subset D(X)$ and $\shB \subset D(Y)$ are $S$-linear subcategories, where $X,Y$ are $S$-schemes, with structural maps $a_X \colon X \to S$ and $a_Y \colon Y \to S$. Assume the fiber square for $X\times_S Y$ is Tor-independent. Then the tensor product of perfect complexes 
	$$\shA^{\rm perf} \otimes_{\Perf(S)} \shB^{\rm perf} \subset \Perf(X \times_S Y)$$
is defined to be the subcategory thickly generated by elements $a_X^* A \otimes a_Y^* B$ for all $A \in \shA^\perf$, $B \in \shB^\perf$. The {\em exterior product of $\shA$ and $\shB$ over $S$} is defined to be
	$$\shA \boxtimes_S \shB : = \shA_Y \cap \shB_X \subset D(X \times_S Y),$$
where $\shA_Y$ is the base-change category of $\shA \subset D(X)$ along $Y \to S$, and $\shB_X$ is the base-change category of $\shB \subset D(Y)$ along $X \to S$.
\end{definition}

Notice if a base-change $\phi \colon T \to S$ is faithful for $X$, then by definition, for any $S$-linear subcategory $\shA \subset D(X)$, the following holds:
	$$\shA^\perf \otimes_{\Perf(S)} \Perf(T) = \shA_T^\perf \subset \Perf(X_T), \qquad
	\shA \boxtimes_{S} D(T) = \shA_T \subset D(X_T),$$
where $\shA_T$ denotes the base-change of category $\shA$ along $T \to S$. Therefore base-change categories can also be expressed by exterior products. 

\begin{remark} \label{rmk:extpd} As in Def. \ref{def:bc-A}, the category $\shA \boxtimes_S \shB$ can also be defined as follows. We denote $\shC^{\rm perf} = \shA^{\rm perf} \otimes_{\Perf(S)} \shB^{\rm perf}$ and consider the completion $\hat{\shC}$ of $\shC^{\rm perf}$ inside $D_{qc}(X \times_S Y)$ under arbitrary direct sums. Then $\shA \boxtimes_S \shB = \hat{\shC} \cap D(X\times_S Y) \subset D(X \times_S Y)$.
\end{remark}

\begin{remark} As mentioned before, the tensor product of perfect complexes $\shA^{\rm perf} \otimes_{\Perf(S)} \shB^{\rm perf}$ agrees with the tensor product of small idempotent-complete stable $\infty$-categories in the sense of \cite{Lur}, \cite{BFN}. If $X$, $Y$ and $X \times_S Y$ are smooth, then $\shA \boxtimes_S \shB = \shA \otimes_{\Perf(S)} \shB$. But in general they are related by 
	$\shA^\perf \otimes_{\Perf(S)} \shB^\perf = (\shA \boxtimes_S \shB) \cap \Perf(X \times_S Y)$ by \cite{Kuz11Bas}, and it is expected that $\shA \boxtimes_S \shB =  {\rm Fun}_{\Perf(S)}^{ex} ( \shA^{\rm perf} \otimes_{\Perf(S)} \shB^{\rm perf}, D(S))$ (for their stable $\infty$-enhancements) by the spirit of \cite{BNP}.
\end{remark}

\begin{lemma}[Associativity] \label{lem:ass:tensor} Assume $X,Y,Z$ are $S$-schemes such that the fiber squares for fiber products $X \times_S Y$, $Y \times_S Z$, $X \times_S Z$ are all Tor-independent. Let $\shA \subset D(X)$, $\shB \subset D(Y)$, $\shC \subset D(Z)$ be $S$-linear admissible subcategories. Then we have a canonical identification
	\begin{align*}
	 (\shA^\perf \otimes_{\Perf(S)} \shB^\perf)  \otimes_{\Perf(S)}\shC^\perf = \shA^\perf  \otimes_{\Perf(S)} ( \shB^\perf \otimes_{\Perf(S)} \shC^\perf ), 
	 \end{align*}
of subcategories of $\Perf(X \times_S Y \times_S Z)$, and
	\begin{align*}
	 (\shA \boxtimes_S \shB) \boxtimes_S \shC = \shA \boxtimes_S (\shB \boxtimes_S \shC) \subset D(X \times_S Y \times_S Z).
	\end{align*}
\end{lemma}
\begin{proof} The statement for perfect complexes holds since they are both the subcategory of $\Perf(X \times_S Y \times_S Z)$ thickly generated by elements of the form
	$$A \otimes B \otimes C, \qquad A \in \shA^\perf, B \in \shB^\perf, C\in \shC^\perf.$$
Then by above definition and Rmk. \ref{rmk:extpd}, the statement for exterior products holds since they are obtained by taking completion of the same above category in $D_{qc}(X \times_S Y \times_S Z)$ and then taking intersection with bounded coherent category $D(X \times_S Y \times_S Z)$.
\end{proof}

\begin{lemma} \label{lem:tensor:bc} Let $S_k, T_k$ be schemes, $k=1,2$, $X_k$ be $S_1 \times T_k$ schemes, $\shA_k \subset D(X_k)$ be $S_1 \times T_k$-linear admissible subcategory, $Y_k$ be $S_2 \times T_k$-schemes, $\shB_k \subset D(Y_k)$ be $S_2 \times T_k$-linear admissible subcategory. Assume that the fiber squares for the fiber products $X_1 \times_{S_1} X_2$, $Y_1 \times_{S_2} Y_2$, $X_1 \times_{T_1} Y_1$, $X_2 \times_{T_2} Y_2$, and 
	$$ (X_1 \times_{S_1} X_2) \times_{T_1 \times T_2} Y_1 (\times_{S_2} Y_2) = (X_1 \times_{T_1} Y_1) \times_{S_1 \times S_2} (X_2 \times_{T_2} Y_2) = : Z$$
are all Tor-independent. Then there is a canonical identification of subcategories
	\begin{align*}
	&(\shA_1^\perf \otimes_{\Perf(S_1)} \shA_2^\perf) \otimes_{\Perf(T_1 \times T_2)} (\shB_1^\perf \otimes_{\Perf(S_2)} \shB_2^\perf)  \\ 
	=&(\shA_1^\perf \otimes_{\Perf(T_1)} \shB_1^\perf)  \otimes_{\Perf(S_1 \times S_2)}(\shA_2^\perf \otimes_{\Perf(T_2)} \shB_2^\perf) 
	 \end{align*}
of $\Perf(Z)$, and a canonical identification of subcategories of $D(Z)$:
	$$(\shA_1 \boxtimes_{S_1} \shA_2) \boxtimes_{T_1 \times T_2} (\shB_1 \boxtimes_{S_2} \shB_2) = (\shA_1 \boxtimes_{T_1} \shB_1) \boxtimes_{S_1 \times S_2} (\shA_2 \boxtimes_{T_2} \shB_2).$$
\end{lemma}

\begin{proof} Similar to the proof of Lem. \ref{lem:ass:tensor}, the statement follows directly from the fact that perfect complexes versions of both sides are generated by the same class of elements.
\end{proof}

 Let $X,Y$ be $S$-schemes, and $\shA \subset D(X)$, $\shB \subset D(Y)$ be $S$-linear admissible subcategories, with $S$-linear semiorthogonal decompositions $\shA = \langle \shA_1, \ldots, \shA_m\rangle$ and $\shB = \langle \shB_1, \ldots, \shB_n \rangle$, $m,n \ge 1$. Assume the following technical condition holds: the projection functors $D(X) \to \shA$, $D(X) \to {}^\perp \shA$, $D(Y) \to \shB$, $D(Y) \to {}^\perp\shB$, $\shA \to \shA_i$, $\shB \to \shB_j$ are all of finite cohomological amplitudes. This condition is automatically satisfied if $X,Y$ are smooth, which will be the case for all applications of the following proposition in this paper.
 
\begin{proposition}[{\cite[\S 5.5]{Kuz11Bas}}] \label{prop:product} In the above situation, we assume the square for fiber product $X \times_S Y$ is Tor-independent. Then there is an $S$-linear semiorthogonal decomposition
	$$\shA \boxtimes_S \shB = \big\langle \shA_i \boxtimes_S \shB_j \big\rangle_{1\le i \le m, 1 \le j \le n},$$
where the order of the semiorthogonal sequence is any order $\{(i,j)\}$ extending the natural partial order of $\{i \mid 1 \le i \le m\}$ and $\{j \mid 1 \le j \le n\}$, and similarly:
	$$\shA^\perf \otimes_{\Perf(S)} \shB^\perf = \big\langle \shA_i^\perf \otimes_{\Perf(S)} \shB_j^\perf \big\rangle_{1\le i \le m, 1 \le j \le n}.$$
\end{proposition}

\begin{proof} This follows directly from applying {\cite[\S 5.5]{Kuz11Bas}} to $D(X) = \langle \shA, {}^\perp \shA  \rangle =  \langle \shA_1, \ldots, \shA_m, {}^\perp \shA \rangle$ and $D(Y) = \langle \shB, {}^\perp \shB \rangle = \langle \shB_1, \ldots, \shB_n, {}^\perp \shB \rangle$.
\end{proof}


\subsection{Lefschetz varieties and categories} \label{sec:lef}
Lefschetz decompositions introduced by Kuznetsov in \cite{Kuz07HPD} (see also \cite{Kuz08Lef, JLX17, P18, KP18}) play an essential role in HPD theory. A Lefschetz decomposition is a special kind of semiorthogonal decomposition which behaves well under a given autoequivalence $T$. The key example of interest will be when $\shA$ is a $\PP(V)$-linear category and $T = \otimes \sO(1)$. For this section we assume $\shA$ is an admissible subcategory of some projective variety over the fixed base-scheme $B$.
 
\begin{definition}[{\cite{Kuz07HPD}}] Let $\shA$ be an admissible category and $T$ be an autoequivalence of $\shA$. A {\em right Lefschetz decomposition} of $\shA$ is a semiorthogonal decomposition of the form
	\begin{equation}\label{lef:A} 
	\shA = \langle \shA_0, T(\shA_1), \ldots, T^{i-1}(\shA_{m-1})\rangle,
	\end{equation}
	where $\shA_0 \supset \shA_1 \supset \cdots \supset \shA_{m-1}$ is a descending sequence of admissible subcategories. Dually, a {\em left Lefschetz decomposition} of $\shA$ is a semiorthogonal decomposition
	\begin{equation}\label{duallef:A} 
	\shA = \langle T^{1-m} \shA_{1-m}, \ldots, T^{-1} \shA_{-1} , \shA_{0}\rangle,
	\end{equation}	
	where $\shA_{1-m} \subset \cdots \subset \shA_{-1} \subset \shA_{0}$ is an ascending sequence of admissible subcategories. 
	\end{definition}
Here $T^k$ denotes the $k$-fold self-composition of $T$ for $k \ge 1$, $T^0 := \Id$, and $T^{-k} := (T^{-1})^k$ for $k \ge 1$.	 $T^k(\shB)$ denotes the image of a subcategory $\shB\subset \shA$ under the $T^k$. 

A Lefschetz decomposition for $\shA$ is totally determined by the component $\shA_0$.
\begin{lemma}[\cite{Kuz08Lef}, Lem. 2.18] \label{lem:lef:A0}
\begin{enumerate} 
	\item Assume there is a right Lefschetz decomposition (\ref{lef:A}) of $\shA$, then $\shA_k = {}^\perp (T^{-k} \shA_0) \cap \shA_{k-1}$ for $k = 1,2, \ldots, m-1$.
	\item Assume there is a left Lefschetz decomposition (\ref{duallef:A}) of $\shA$, then $\shA_{-k} = (T^k \shA_0)^\perp \cap \shA_{1-k}$ for $k=1,2,\ldots, m-1$.
	\end{enumerate}
\end{lemma}

Assume $\shA$ is a Lefschetz category, denote $\foa_k: = \shA_{k+1}^\perp \cap \shA_k$, the right orthogonal of $\shA_{k+1}$ inside $\shA_{k}$ for $0 \le k \le m-1$. Then $\foa_k$'s are admissible subcategories, and for $0 \le k \le m-1$, 
	$$\shA_k = \langle \foa_{k}, \foa_{k+1}, \ldots, \foa_{m-1}\rangle.$$
Dually let $\foa_{-k}: = {}^\perp \shA_{-k-1} \cap \shA_{-k}$ be the left orthogonal of $\shA_{-k-1}$ inside $\shA_{-k}$ for $0 \le k \le m-1$. Then $\foa_{-k}$'s are also admissible subcategories, and for $0 \le k \le m-1$,
	$$\shA_{-k} = \langle \foa_{1-m}, \ldots, \foa_{-1-k}, \foa_{-k}\rangle.$$
Let $\alpha_0: \shA_0 \to D(X)$ be the inclusion functor, $\alpha_0^*: D(X) \to \shA_0$ be its left adjoint and $\alpha_0^!: D(X) \to \shA_0$ be its right adjoint.

\begin{lemma} \label{lem:sod:A_0} 
\begin{enumerate} [leftmargin=*]
	\item Assume there is a right Lefschetz decomposition (\ref{lef:A}) of $\shA$, then for any integer $0 \le k\le m-1$, $\alpha_0^*$ is fully faithful on $T^{k+1}(\foa_{k})$. If we denote 
	$$\foa_k' := \alpha_0^*(T^{k+1}(\foa_{k})) \subset \shA_0 \quad \text{for}\quad k = 0,1,\ldots, m-1,$$
then $(\foa_0', \foa_1', \ldots, \foa_{m-1}')$ is a semiorthogonal sequence, and for $k=0,\ldots, m-1$,
	$$\langle \foa_0', \ldots, \foa_{k-1}', T(\shA_1), \ldots, T^k(\shA_{k})\rangle = \langle T(\shA_0), \ldots, T^k(\shA_{k-1})\rangle,$$
(where we regard $\shA_{m}=0$.) In particular, $\shA_0 = \langle \foa_0', \foa_1', \ldots, \foa_{m-1}'\rangle$.
	\item Assume there is a left Lefschetz decomposition (\ref{duallef:A}) of $\shA$, then for any integer $k$, $0 \le k\le m-1$, $\alpha_0^!$ is fully faithful on $T^{-k}(\foa_{-k+1})$. If we denote 
	$$\foa_{-k}' := \alpha_0^!(T^{-1-k}(\foa_{-k})) \subset \shA_0 \quad \text{for}\quad k = 0,1,\ldots, m-1,$$	
then $(\foa_{1-m}', \ldots, \foa_{-1}', \foa_{-0}')$ is a semiorthogonal sequence, and and for $k=0,\ldots, m-1$, 
	$$\langle T^{-k}(\shA_{-k}), \ldots, T^{-1}(\shA_{-1}), \foa_{1-k}' \ldots, \foa_{-0}' \rangle = \langle T^{-k}(\shA_{1-k}), \ldots, T^{-1}(\shA_{0})\rangle.$$
(where we regard $\shA_{-m}=0$.) In particular, $\shA_0 = \langle \foa_{1-m}', \ldots, \foa_{-1}', \foa_{-0}' \rangle$. 
	\end{enumerate} 	
\end{lemma}

\begin{proof} See \cite[Lem. 4.3]{Kuz07HPD},  \cite[Lem. 2.12]{JLX17}.
\end{proof}

\begin{definition} A $\PP(V)$-linear category $\shA$ is called a {\em Lefschetz category} with {\em Lefschetz center} $\shA_0 \subset \shA$ if $\shA$ admits right Lefschetz decomposition (\ref{lef:A}) and left Lefschetz decomposition (\ref{duallef:A}) with respect to the autoequivalence $T = \otimes \sO_{\PP(V)}(H) \colon \shA \to \shA$, where $\shA_k$ is determined by $\shA_0$ through the relation of Lem. \ref{lem:lef:A0}. The number $m \ne 0$ in (\ref{lef:A}) (such that $\shA_{m-1} \ne 0$) is called the {\em length} of the Lefschetz category $\shA$, and denoted by
	$${\rm length} (\shA) = m.$$ 
\end{definition}

Given a right Lefschetz decomposition of $\shA$, there is a canonical left Lefschetz category of the same center, and vice versa, see \cite[Lem. 2.18]{Kuz08Lef}. The next lemma as an extension of \cite[Lem. 2.18]{Kuz08Lef}, shows the relation between right and left Lefschetz decompositions. This implies that for $\shA$ an admissible subcategory of some smooth projective variety over $S$ (which is always assumed to be true in this paper), then $\shA$ is a Lefschetz category if and only if it admits a right Lefschetz decomposition (\ref{lef:A}), or equivalently if it admits a left Lefschetz decomposition (\ref{duallef:A}), with the same center.

\begin{lemma}\label{lem:lef&dual}
	\begin{enumerate} [leftmargin=*]
	\item Given a right Lefschetz decomposition (\ref{lef:A}) of $\shA$, then there is a canonical left Lefschetz decomposition of the form (\ref{duallef:A}) given by
		$$\shA_{-0} = \shA_0,\quad \shA_{-k} := (T^k(\shA_0))^\perp \cap \shA_{-k+1} \quad \text{for} \quad k=1,2,\ldots, m-1,$$ 
Moreover, if the Serre functor $S_0$ of $\shA_0$ exists, then $\shA_{-k} \subset \shA_0$ can also be given by
		\begin{align}\label{sod:A-k:S_0}
	\shA_{-k} = \langle S_0 (\foa_{k}'), \ldots, S_0 (\foa_{m-1}') \rangle, \quad k=0,1,\ldots, m-1.
		\end{align}
The building components $\foa_{-k}$'s are determined by $\foa_k$'s via
	\begin{align*}
	\foa_{-0} &= \alpha_0^*(T(\foa_{0}))  = T(\foa_0);\\
	\foa_{-k} & = \LL_{\langle \foa_{1-k}, \ldots, \foa_{-0} \rangle}\, \foa_{k}' = \LL_{\langle \foa_{0}', \ldots, \foa_{k-1}' \rangle} \, \foa_{k}' 	\quad \text{for} \quad k = 1,2,\ldots,m-1.
	\end{align*}
In particular one has 
	$\foa_{-k} \simeq \foa_k' \simeq \foa_k,$ \text{and} $\shA_{-k} \simeq \langle \foa_{k}', \foa_{k+1}', \ldots, \foa_{m-1}' \rangle$
	\medskip
	\item Dually, given a left Lefschetz decomposition (\ref{duallef:A}) of $\shA$, then there is a canonical right Lefschetz decomposition of the form (\ref{lef:A}) given by
		$$\shA_0 = \shA_{-0}, \quad  \shA_{k} := {}^\perp (T^{-k} \shA_0) \cap \shA_{k-1},\quad \text{for} \quad k=1,2,\ldots, m-1,$$ 
Moreover, if the Serre functor $S_0$ of $\shA_0$ exists, then $\shA_{k} \subset \shA_0$ can also be given by
		\begin{align}\label{sod:Ak:S_0}
		\shA_{k} = \langle S_0^{-1} (\foa_{1-i}'), \ldots, S_0^{-1} (\foa_{-k}') \rangle, \quad k=0,1,\ldots, m-1.
		\end{align}	
The building components $\foa_{k}$'s are determined by $\foa_{-k}$'s via
	\begin{align*}
	\foa_0 &= \alpha_0^!(T^{-1}\foa_{-0})  = T^{-1}(\foa_{-0});\\
	\foa_{k} & = \RR_{\langle \foa_{0}, \ldots, \foa_{k-1} \rangle}\, \foa_{-k}' = \RR_{\langle \foa_{1-k}', \ldots, \foa_{-0}' \rangle} \, \foa_{-k}' 	\quad \text{for} \quad k = 1,2,\ldots,m-1.
	\end{align*}
In particular one has 
	$\foa_{k} \simeq \foa_{-k}' \simeq \foa_{-k},$ \text{and} $\shA_{k} \simeq \langle \foa_{1-m}', \ldots, \foa_{-1-k}', \ldots, \foa_{-k}' \rangle.$
	\end{enumerate} 
\end{lemma}
\begin{proof}
\noindent \text{(1)}. Let $\shA_k$'s be defined by (\ref{sod:A-k:S_0}), then they satisfy $\shA_{-k} \simeq \langle \foa_{k}', \foa_{k+1}', \ldots, \foa_{m-1}' \rangle$ by definition. Following \cite[Lem. 2.18]{Kuz08Lef}, we claim for $k = 0,1,\ldots, m-1$, the sequence $\shA_{-k}, T(\shA_{-k+1}), \ldots, T^k(\shA_0)$ is a semiorthogonal sequence, and 
	\begin{equation} \label{eqn:Lef:pm}
	\langle \shA_{-k}, T(\shA_{-k+1}), \ldots, T^k(\shA_0) \rangle = \langle \shA_0, T(\shA_1), \ldots, T^k(\shA_k) \rangle.
	\end{equation}
Then the lemma follows from the case $k=m-1$. The claim can be shown inductively. For $k=0$ this is trivial. Assume it is true for $k-1$, then consider the semiorthogonal collection
	$$\langle T(\shA_{-k+1}), \ldots, T^k(\shA_0) \rangle =T \langle \shA_0, T(\shA_1), \ldots, T^{k-1}(\shA_{k-1})\rangle = \langle T(\shA_0), \ldots, T^k(\shA_{k-1})\rangle.$$
Then by Lem. \ref{lem:sod:A_0}, the right hand side is equal to
	$\langle \foa_0' , \ldots, \foa_{k-1}', T(\shA_1), \ldots, T^k(\shA_{k})\rangle. $
Consider the semiorthogonal decomposition of $\shA_0$:
	\begin{align} 
	\shA_0 & \notag= \langle \foa_0', \ldots, \foa_{k-1}', \foa_{k}', \ldots,  \foa_{m-1}'\rangle  =  \langle S_0( \foa_{k}'), \ldots,  S_0(\foa_{m-1}'), ~  \foa_0', \ldots, \foa_{k-1}' \rangle \\
	& =  \langle \shA_{-k}, ~  \foa_0', \ldots, \foa_{k-1}' \rangle . \label{eqn:A_0,A_-k}
	\end{align}
Hence $\shA_{-k}$ together with $ \foa_0', \ldots, \foa_{k-1}'$, and $T(\shA_1), \ldots, T^k(\shA_{k})$ forms a semiorthogonal decomposition of $\langle \shA_0, T(\shA_1), \ldots, T^k(\shA_{k})\rangle$. Therefore (\ref{eqn:Lef:pm}) holds. In particular if $k=m-1$, then this implies that $\langle T^{1-m}(\shA_{1-m}), \ldots, T^{-1}(\shA_{-1}), \shA_0 \rangle = \shA$ forms a left Lefschetz decomposition. Then by Lem. \ref{lem:lef:A0}, $\{\shA_{-k}\}_{0 \le k \le m-1}$ automatically satisfies $\shA_{-k} = (T^k(\shA_0))^\perp \cap \shA_{-k+1}$. Denote $\foa_{-k} = {}^\perp \shA_{-k-1} \cap \shA_{-k}$ as usual, then from (\ref{eqn:A_0,A_-k}) one has
	$$\langle \foa_0', \ldots, \foa_{k}'  \rangle = \langle  \foa_{-k},\ldots, \foa_{-0}\rangle \quad \text{for} \quad k=0,1,2,\ldots, m-1.$$
The description of $\foa_{-k}$ in terms of mutations of $\foa_{k}'$ follows directly from induction.

\medskip
\noindent \text{(2)}. Dually, define $\shA_k$ by (\ref{sod:Ak:S_0}), then one can inductively show for $k=1,\ldots,m-1$,
		$$\langle T^{-k}(\shA_0), T^{-k+1}(\shA_1), \ldots, T^{-1}(\shA_{k-1}), \shA_k \rangle = \langle T^{-k}(\shA_{-k}), \ldots, T^{-1}(\shA_{-1}), \shA_{-0}\rangle.$$
The lemma follows from $k=m-1$.
\end{proof}

\subsubsection{$\PP(V)$-linear Lefschetz category} In HPD theory we will mainly consider {$\PP(V)$-linear Lefschetz categories}.

\begin{definition} A {\em $\PP(V)$-linear Lefschetz category $\shA$} is a $\PP(V)$-linear $\shA$ with a Lefschetz structure with respect to the autoequivalence $T = - \otimes \sO_{\PP(V)}(1)$. It is common to write:
	$$\shA = \langle \shA_0, \shA_1(1), \ldots, \shA_{m-1}(m-1)\rangle,$$
where $\shA_*(k)$ denotes the image of $\shA_*$ under the autoequivalence $\otimes \sO_{\PP(V)}(k)$, for $k \in \ZZ$. A variety $X$ with a morphism $X \to \PP(V)$ is called a {\em Lefschetz variety} if $\shA = D(X)$ is a $\PP(V)$-linear Lefschetz category. 
\end{definition}

In \cite{Kuz07HPD} it is required that a $\PP(V)$-linear Lefschetz category should satisfies
	$$ {\rm length} (\shA) < \rank V = :N.$$
This condition is called {\em moderate} in \cite{P18}. We will also stick to this convention and require the moderateness of all Lefschetz categories in this paper. In fact, the next result, which is a generalization of \cite[Cor. 6.19]{P18}, shows that a non-moderate Lefschetz category can always be refined to be a moderate one:

\begin{lemma} \label{lem:non-moderate Lef}  \label{lem:lef:nonmoderate} Assume $\shA$ is a $\PP(V)$-linear Lefschetz category of length $m$ with respect to the action $T = \otimes \sO_{\PP(V)}(1)$, with Lefschetz center $\shA_0$ and components $\shA_k$.
	\begin{enumerate}
	\item Then length $m$ of $\shA$ satisfies $m \le N$;
	\item If $m = N$, then the action functor $\act \colon \shA \boxtimes D(\PP(V)) \to \shA$ is fully faithful on $\shA_{N-1} \boxtimes D(\PP(V))$. Moreover, if we denote 
	$$\shA'_k = \shA_k \cap \shA_{N-1}^{\perp} \quad \text{for} \quad 0 \le k \le N-2,$$
then the image $\act(\shA_{N-1} \boxtimes D(\PP(V))) \subset \shA$ is right orthogonal to the image $\act(\shA'_{k} \boxtimes D(\PP(V))) \subset \shA$ for all $0 \le k \le N-2$.
	\item If $m=N$, denote the right orthogonal of $\act(\shA_{N-1} \boxtimes D(\PP(V)))$ inside $\shA$ by
		$$\shA' := \act(\shA_{N-1} \boxtimes D(\PP(V)))^{\perp} \cap \shA \subset \shA.$$
Then $\shA'$ is a {\em moderate} $\PP(V)$-linear Lefschetz category, with Lefschetz center
	$$\shA'_0 = \shA_0 \cap \shA_{N-1}^{\perp},$$
and Lefschetz components $\shA'_k = \shA_k \cap \shA_{N-1}^{\perp}$ for $k \ge 0$. 
	\end{enumerate}
\end{lemma}

\begin{proof} For $(1)$ this is \cite[Cor. 6.19 (1)]{P18}; we present a proof here without using argument on empty sets. Assume $\shA \subset D(X)$ for a variety $X \to \PP(V)$, and consider the Koszul complex $K^\bullet$ on $X$ for the canonical section $s \in H^0(X, V^\vee \otimes \sO_X(1)) = \Hom(V^\vee, H^0(X, \sO_X(1)))$ corresponding to the linear system $V^\vee$. Since the linear system is base-point free, $K^\bullet$ is exact, hence $K^{\le -1} \simeq \sO_X$, therefore for all $C_1, C_2 \in \shA$, we have
	\begin{align*}
		\Hom(C_1,C_2) \simeq \Hom(C_1, C_2 \otimes K^{\le -1}).
	\end{align*}
Notice $K^{\bullet}$ is given by $K^r = 0$ for $r \ge 0$, and
	$K^{-r} = \wedge^{r}(V^\vee \otimes \sO_X(-1)) = \wedge^{r}V^\vee \otimes \sO_X(r)$ for  r $\ge 0.$
Therefore for any $C \in \shA_k$, $k \ge N$, $r \ge 1$,
	$$\Hom(C,C) \simeq \Hom(C,C \otimes K^{-r}) = \Hom(C, C(-r)) \otimes  \wedge^{r}V =0, $$
since $C(-r) \in \shA_k^{\perp}$. This implies $C=0$. Hence $\shA_k = 0$ for $k \ge N$.

For $(2)$, notice that the projectivization of the inclusion $\sO_{X}(-1) \hookrightarrow V \otimes \sO_X$, induced from the inclusion $\sO_{\PP(V)}(-1) \hookrightarrow V \otimes \sO_{\PP(V)}$ on $\PP(V)$,
induces a tautological inclusion $\Gamma \colon X \hookrightarrow X \times \PP(V)$ (which is also the graph morphism), with normal bundle $\sO_X(-1) \boxtimes \Omega^1_{\PP(V)}(1)$. The pull-back functor $\Gamma^*$, when restricted to $\shA \boxtimes D(\PP(V))$, is nothing but the action functor:	
	$$\Gamma^*|_{\shA \boxtimes D(\PP(V))} = \act  \colon \shA \boxtimes D(\PP(V)) \to \shA.$$

By Koszul resolution for the regular embedding $\Gamma$, for any $C_1, C_2 \in \shA$, $F_1, F_2 \in D(\PP(V))$, and adjunction $1 \to \Gamma_* \, \Gamma^*$, the cone of the natural map 
	$$\Hom_{\shA \boxtimes \PP(V)}(C_1\boxtimes F_1, C_2 \boxtimes F_2) \to \Hom_{\shA}(\Gamma^*(C_1\boxtimes F_1), \Gamma^*(C_2 \boxtimes F_2))$$
is an iterated cone of the $\Hom$ spaces:
	\begin{align*}
	&\Hom_{\shA \boxtimes \PP(V)}\big(C_1 \boxtimes F_1, C_2(-r) \boxtimes F_2 \otimes \Omega^r(r) \big)[-r]  \\
	= &  \Hom_{\shA}\big(C_1,C_2(-r) \big) \otimes \Hom\big(F_1, F_2 \otimes \Omega^r(r) \big)[-r], 	&1\le r \le N-1.
	\end{align*}
Notice if $r=0$ then above $\Hom$ space is nothing but $\Hom_{\shA \boxtimes \PP(V)}(C_1\boxtimes F_1, C_2 \boxtimes F_2)$.

For any $C_1, C_2 \in \shA_{N-1}$, $F_1, F_2 \in D(\PP(V))$ above $\Hom$ space vanishes for $1 \le r \le N-1$, since $\Hom(\shA_{N-1}(i), \shA_i) = 0$ for all $1 \le i \le N-1$ by definition of Lefschetz decomposition. Therefore $\Gamma^*$ is fully faithful on $\shA_{N-1} \boxtimes D(\PP(V))$. Furthermore, for any $C_1 \in \shA_{N-1}$, $C_2 \in \shA_k'$, $0 \le k \le N-2$, $F_1, F_2 \in D(\PP(V))$, the above $\Hom$ space vanishes for all $0 \le r \le N-1$, by the definition of $\shA_k'$. Therefore the semiorthogonal statement of $(2)$ follows.
	
For $(3)$, from 
	$\shA_k' = \langle \shA_k, \shA_{N-1} \rangle,$
and the semiorthogonal statement of $(2)$, the Lefschetz decomposition of 
	$\shA = \langle \shA_0, \shA_1(1), \ldots, \shA_{N-1}(N-1)\rangle$
 can be mutated into
	$$\shA = \langle \shA_0', \shA_1'(1), \ldots, \shA_{N-2}'(N-2), ~ \Gamma^*(\shA_{N-1} \boxtimes D(\PP(V)))\rangle,$$
therefore the right orthogonal of the $\PP(V)$-linear category $\Gamma^*(\shA_{N-1} \boxtimes D(\PP(V)))$ of $\shA$:
	$$\shA' = \langle \shA_0', \shA_1'(1), \ldots, \shA_{N-2}'(N-2) \rangle $$
is a $\PP(V)$-linear Lefschetz category, with Lefschetz components $\shA_0' \supset \shA_1' \supset \cdots \supset \shA_{N-2}'$.
\end{proof}

\subsection{Homological projective duality}\label{sec:HPD}

\begin{definition}\label{def:HPD}
Let $X$ be a variety with morphism $X \to \PP(V)$. The {\em universal hyperplane} $\shH_X  = \shH(X) / \PP(V)$ is defined to be the variety
	$$\shH_X =  \shH(X) / \PP(V) := X \times_{\PP V} Q \xrightarrow{\delta_{\shH}} X \times \PP(V^\vee).$$
If $\shA \subset D(X)$ is a $\PP(V)$-linear admissible subcategory, then the {\em universal hyperplane} $\shH_{\shA} = \shH(\shA)/\PP(V)$ is defined to be the admissible subcategory:
	$$\shH_{\shA} : = (\shA)_{\shH_X} = \shA \boxtimes_{\PP(V)} D(Q) \subset D(\shH_X),$$
where $(\shA)_{\shH_X}$ denotes the base-change of $\shA$ along $\shH_X \to \PP(V)$.
\end{definition}

If follows directly that $\delta_{\shH*} \colon D(\shH_X) \to D(X \times \PP(V^\vee))$ induces a functor $\delta_{\shH*} \colon \shH_\shA \to \shA \boxtimes D(\PP(V^\vee))$, with left adjoint given by the restriction of the functor $\delta_{\shH}^* \colon D(X \times \PP(V^\vee)) \to D(\shH_X)$ to subcategory $\shA \boxtimes D(\PP(V^\vee))$, still denoted by $\delta_{\shH}^* \colon  \shA \boxtimes D(\PP(V^\vee)) \to \shH_\shA$. Composed with projection $\pi \colon \shH_X \to X$, one sees that the functor $\pi_{*} \colon D(\shH_X) \to D(X)$ induces $\pi_{*} \colon \shH_{\shA} \to \shA$. It also follows that $\shH_{\langle \shA, \shB \rangle} = \langle \shH_{\shA}, \shH_{\shB}\rangle$. See \S \ref{sec:app:hyp} for more details.

\begin{definition} \label{def:HPDcat}
Let $\shA$ be a $\PP(V)$-linear Lefschetz subcategory with Lefschetz center $\shA_0$. Then the {\em HPD category $\shA^\hpd$} of $\shA$ over $\PP(V)$ is defined to be
	$$\shA^\hpd  = (\shA/\PP(V))^\hpd =  \{C \in \shH_{\shA}  \mid \delta_{\shH*} \,C \in \shA_0 \boxtimes D(\PP(V^\vee))\} \subset \shH_{\shA}.$$
If $\shA = D(X)$ for a variety $X \to \PP(V)$, and there exists a variety $Y$ with $Y \to \PP(V^\vee)$, and a Fourier-Mukai kernel $\shP \in D(Y \times_{\PP(V^\vee)} \shH_X)$ such that the $\PP(V^\vee)$-linear Fourier-Mukai functor $\Phi_{\shP}^{Y \to \shH_X} \colon D(Y) \to D(\shH)$ induces $D(Y) \simeq D(X)^\hpd$, then $Y$ is called the {\em HPD variety} of $X$.
\end{definition}


\begin{lemma}[{\cite{Kuz07HPD,JLX17,P18}}] \label{lem:sodH} There is a $\PP(V^\vee)$-linear semiorthogonal decomposition
	$$\shH_{\shA} = \big \langle \shA^\hpd, ~~\delta_{\shH}^* (\shA_1(1) \boxtimes D(\PP(V^\vee))), \ldots,  \delta_{\shH}^* (\shA_{m-1}(m-1) \boxtimes D(\PP(V^\vee))) \big \rangle.$$
\end{lemma}



The output of HPD theory due to Kuznetsov, especially the fundamental theorem of HPD for linear sections, is summarized as follows, see {\cite{Kuz07HPD,JLX17,P18} for more details. Denote $\gamma \colon \shA^\hpd \hookrightarrow \shH_{\shA}$ the inclusion functor, and $\gamma^*$ its left adjoint as usual.

\begin{theorem}[Kuznetsov] \label{thm:HPD}
\begin{enumerate}[leftmargin=*]
	\item (Decomposition of the dual) The HPD category $\shA^\hpd$ is a Lefschetz category with respect to $-\otimes \sO_{\PP(V^\vee)}(H')$ of length:
		$$n : = N - 1 - \min \{i \ge 0 \mid \shA_i = \shA_0\},$$
	 Lefschetz center $\shA^\hpd_0$ such that $\pi_* \, \gamma \colon \shA_0^\hpd \simeq \shA_0$ and $\gamma^* \, \pi^* \colon \shA_0 \simeq \shA_0^\hpd$, and Lefschetz components
		$$\shA^{\hpd}_j = \langle \gamma^* \pi^* \foa_{N+1-n}', \ldots, \gamma^* \pi^* \foa_{N-2+j}' \rangle, \qquad  -n+1 \le j  \le 0.$$
	Namely, there is a left Lefschetz decomposition of $\shA^\hpd$ with respect to $\otimes \sO_{\PP(V^\vee)}(H')$
		$$\shA^\hpd = \langle \shA^\hpd_{1-n} (1-n), \ldots, \shA_{-1}^\hpd(-1), \shA_0^\hpd \rangle$$
	\item (Duality) There is a $\PP(V)$-linear equivalence of Lefschetz categories $\shA \simeq (\shA^\hpd)^\hpd$.
	\item (Fundamental theorem of HPD) For a generic linear subbundle $L \subset V^\vee$ of rank $\ell$, there is a common triangulated category $\shC_L$ such that are semiorthogonal decompositions 
		\begin{align*}
		\shA_{\PP(L^\perp)} & = \langle \shC_L, ~~, \shA_{\ell}(1), \ldots, \shA_{m-1}(m-\ell) \rangle, \\
		\shA^\hpd_{\PP(L)} & = \langle \shA^\hpd_{1-n}(-n+\ell -N), \ldots, \shA^\hpd_{\ell-N}(-1), \shC_L\rangle.	
		\end{align*}	
\end{enumerate}
\end{theorem}

\begin{remark} In \cite{Kuz07HPD} or its noncommutative version \cite{P18}, the first statement (1) is proved after they have proved the fundamental theorem of HPD, and Kuznetsov posed the question in \cite{Kuz07HPD} that whether there is a direct proof of (1). This is solved in \cite{JLX17} where (1) is proved directly using ``chess game".
\end{remark}

\begin{remark} There is a way to reduce the burden of remembering the indices for the HPD categories, introduced in \cite{JLX17}: if we introduce the ``cohomological convention"  $\shB^{j} \simeq \langle \foa_{0}', \ldots, \foa_{j-1}' \rangle$, then $\shB^{j} \simeq \shA^{\hpd}_{j+1-N}$, and we have the following simple expressions
	$$\shA^{\hpd} = \big \langle \shB^{j} (j-(N-1)) \big \rangle_{1 \le j \le N-1} \quad \text{and} \quad \shA^{\hpd}_{\PP(L)} = \big \langle  \langle \shB^{j} (j-\ell)\rangle_{1 \le j \le \ell-1}, ~\shC_L \big \rangle.$$
\end{remark}

\section{Geometric operations on admissible subcategories} \label{sec:app:gem:cat}
In this section we discuss how basic geometric operations -- projective bundle, (generalized) universal hyperplane and blowing up, etc -- can be performed on categories. These results allow one to extend easily the constructions/arguments from commutative settings to noncommutative ones. The blowing-up formula Thm. \ref{thm:bl} seems to be absent in literatures even in the commutative case, namely Cor. \ref{cor:bl:lci} for possibly singular local complete intersection centers. As usual we stick to admissible subcategories of schemes, and the readers should have no difficulties to translate the content into $dg$-setting or (stable) $\infty$-setting. 

\subsection{Projective bundle} \label{sec:app:proj_bd} Let $S$ be a smooth $B$-scheme, $E$ be a vector bundle of rank $r$ on $S$ and $\pi \colon \PP_S(E) \to S$ be the projection. Let $X$ be a proper $S$-scheme, $i_{\shA} \colon \shA \hookrightarrow D(X)$ be an inclusion of $S$-linear admissible subcategory, then $\shA^{\perf} = \shA \cap \Perf(X) \subset \Perf(X)$ is $S$-linear admissible subcategory. Denote the base-change categories along $\PP_S(E) \to S$ by:
	$$\PP_\shA (E) = \shA_{\PP_S(E)} = \shA \boxtimes_{S} D(\PP_S(E)), \qquad \PP_{\shA}^{\perf}(E) = \shA^{\perf} \otimes_{\Perf S} \Perf(\PP_S(E)). $$
We will also denote $\PP_{\shA} = \PP_\shA (E)$ and $\PP_{\shA}^\perf = \PP_{\shA}^{\perf}(E)$ if there is no confusion. Therefore it follows directly from Prop. \ref{prop:product} that if there is $S$-linear semiorthogonal decomposition $D(X) = \langle \shA, \shB\rangle$ then $D(\PP_X(E)) = \langle \PP_\shA, \PP_\shB \rangle$, and $\Perf(\PP_X(E)) = \langle \PP_\shA^\perf, \PP_\shB^\perf \rangle$. Notice that from properties of base-change categories \S \ref{sec:bc} we have a commutative diagrams of $S$-linear functors
	$$
	\begin{tikzcd}
	\PP_\shA \ar[shift left]{d}{\pi_*} \ar[hook]{r}{i_{\PP_\shA}} & D(\PP_X(E))  \ar[shift left]{d}{\pi_*}\\
	\shA \ar[shift left]{u}{\pi^*} \ar[hook]{r}{i_\shA} & D(X) \ar[shift left]{u}{\pi^*}
	\end{tikzcd}
	$$
where $i_{\PP_{\shA}} \colon  \PP_\shA \hookrightarrow D(\PP_X(E))$ denotes the inclusion; similarly for the perfect complexes. Notice the adjoint functors $\pi^* \colon D(X) \rightleftarrows D(\PP_{X}(E)) \colon \pi_*$ induce adjoint functors 
	$\pi^* \colon \shA \rightleftarrows \PP_{\shA}(E) \colon \pi_*,$ still denoted by same notations, by abuse of notations. From Prop. \ref{prop:bcsod}, Orlov's result \cite{O92} translates into:

\begin{theorem}[Orlov's Projective bundle formula, \cite{O92}] \label{thm:app:pr_bd} The functors $\pi^*(-) \otimes \sO(k) \colon \shA \to \PP_{\shA}$ is fully faithful, $k \in \ZZ$, and  there is a $S$-linear semiorthogonal decomposition
	$$\PP_\shA = \langle \pi^*\shA, \pi^*\shA \otimes \sO(1), \ldots, \pi^*\shA \otimes \sO(r-1)\rangle,$$
where $\sO(k)$ denotes the pull-back of $\sO_{\PP_S(E)}(k)$. There is also a similar decomposition for the perfect complexes $\PP_\shA^\perf = \langle \pi^*\shA^\perf, \pi^*\shA^\perf \otimes \sO(1),\ldots, \pi^* \shA^\perf \otimes \sO(r-1)\rangle$.
\end{theorem}

\subsection{Universal hyperplane} \label{sec:app:hyp} Apply the results of last section to $S= \PP(V)$ (where $V$ is a vector bundle over the fixed base scheme $B$ of rank $N$), and $E = \Omega_{\PP(V)}^1\otimes \sO_{\PP(V)}(1) \subset E' = V^\vee \otimes \sO_{\PP(V)}$, Then $\PP_S(E) = Q \subset  \PP(V) \times \PP(V^\vee)$  is nothing but the universal quadric. Now we assume $X$ is a {\em smooth} $S$-scheme and $\shA \subset D(X)$ be an admissible subcategory. Denote $\pi \colon \shH_X = \PP_X(E) \to X$ the projection and
	$\shH_\shA = \PP_\shA(E)$,
then $\shH_\shA$ is the universal hyperplane for $\shA$ (see definiiton \ref{def:HPD}). Then from properties of projective bundles, if there is a $\PP(V)$-linear semiorthogonal decomposition $D(X) = \langle \shA, \shB\rangle$, then $D(\shH_X) = \langle \shH_\shA, \shH_\shB\rangle$. And also there is a $\PP(V)$-linear semiorthogonal decomposition:
	$$\shH_\shA = \langle \pi^*\shA, \pi^* \shA\otimes \sO_{\PP(V^\vee)}(1), \ldots, \pi^* \shA\otimes \sO_{\PP(V^\vee)}(N-2)\rangle \subset D(\shH_X).$$
Denote $\delta \colon \shH_X \hookrightarrow X \times \PP(V^\vee)$ the inclusion and $h\colon \shH_X \to \PP(V^\vee)$ the projection. Then the adjoint functors $\delta^*\colon D(\shH_X) \rightleftarrows D(X) \boxtimes D(\PP(V^\vee)) \colon \delta_*$ and $\pi^* \colon D(X) \rightleftarrows D(\shH_X) \colon \pi_*$ induce adjoint functors on the corresponding subcategories; we still denote the functors by same notations by abuse of notations. Therefore we have a diagram of $\PP(V)$-linear functors:
	$$
	\begin{tikzcd}
	\shA \boxtimes  D( \PP(V^\vee)) \ar[shift left]{d}{\delta^*}  	\ar[hook]{r}{i_\shA \times \Id} 	&	D(X) \boxtimes D( \PP(V^\vee)) \ar[shift left]{d}{\delta^*}  \\
	\shH_\shA \ar[shift left]{d}{\pi_*} \ar[hook]{r}{i_{\PP_\shA}} 	\ar[shift left]{u}{\delta_*}	& 		D(\shH_X)  \ar[shift left]{d}{\pi_*} \ar[shift left]{u}{\delta_*}	\\
	\shA \ar[shift left]{u}{\pi^*} \ar[hook]{r}{i_\shA} & D(X) \ar[shift left]{u}{\pi^*}
	\end{tikzcd}.
	$$
Therefore by Lem. \ref{lem:bc-char}, $\shH_\shA \subset D(\shH_X)$ is also characterized by
	\begin{align*} \shH_\shA  &=  \{C \in  D(\shH_{X})  \mid \delta_{*} \,C \in \shA \boxtimes D(\PP(V^\vee))\} \subset D(\shH_X)	\\
	&=    \{C \in D(\shH_{X})  \mid \pi_* (C\otimes h^* F) \in \shA, \quad \forall F\in D(\PP(V^\vee)) \}  \subset D(\shH_X).
	\end{align*}
\subsubsection{HPD category}\label{sec:HPD:cat} From last section, it follows directly from the definition of HPD category $\shA^\hpd \subset \shH_\shA$ that $\shA^\hpd$ also admits the following descriptions as a subcategory of $D(\shH_X)$:
	 \begin{align*}
	 \shA^\hpd  &=  \{C \in  D(\shH_{X})  \mid \delta_{\shH*} \,C \in \shA_0 \boxtimes D(\PP(V^\vee))\}  \subset  D(\shH_{X}) \\
	 &=  \{C \in D(\shH_{X})  \mid \pi_* (C\otimes h^* F) \in \shA_0, \quad \forall F\in D(\PP(V^\vee)) \} \subset  D(\shH_{X}).
	 \end{align*}
	 
\subsection{Blowing up}\label{sec:bl} Let $S$ be a smooth $B$-scheme and $i \colon Z \hookrightarrow S$ be an inclusion of smooth codimension $r \ge 2$ local complete intersection subscheme. For simplicity assume in this subsection $B$ is a smooth scheme over a field of characteristic zero. Then the normal bundle $N_i$ of $Z\subset S$ is a vector bundle of rank $r$. Denote $\beta \colon \widetilde{S} = \Bl_Z S \to S$ the blowing up of $S$ along $Z$, $j \colon E_Z = \PP(N_i) \hookrightarrow  \widetilde{S}$ the inclusion of exceptional divisor, and $p \colon E_Z \to Z$ the projection. We have a commutative diagram:
	\begin{equation}\label{diag:Bl_Z S}
	\begin{tikzcd}[row sep= 2.6 em, column sep = 2.6 em]
	E_Z \ar{d}[swap]{p} \ar[hook]{r}{j} & \widetilde{S}\ar{d}{\beta} \\
	Z \ar[hook]{r}{i}         & S 
	\end{tikzcd}	
	\end{equation}
Now let $X$ be a smooth proper $S$-scheme, assume $X_Z := X \times_S Z$ is of expected dimension $\dim X - r$, therefore $X_Z \subset X$ is local complete intersection of codimension $r$. By abuse of notations, denote by $\beta \colon \widetilde{X} = \Bl_{X_Z} X \to X$ the blowing up of $X$ along $X_Z$ , by $j \colon E_{X_Z} = \PP_{X_Z}(N_i) \hookrightarrow \widetilde{X}$ the inclusion of exceptional divisor, and by $p\colon E_{X_Z} \to X_Z$ the projection. Then we have the following Tor-independent (fibered) squares:
	$$
	\begin{tikzcd}
	X_Z  \ar{d} \ar[hook]{r}{i}  & X \ar{d} \\
	Z 	\ar[hook]{r}{i} & S 
	\end{tikzcd} \qquad \text{and} \qquad
	\begin{tikzcd}
	X_Z  \ar{d} 	& E_{X_Z} \ar{l}[swap]{p}\ar[hook]{r}{j}	\ar{d} & \widetilde{X} \ar{r}{\beta}	\ar{d}	& X \ar{d}\\
	Z		& E_Z	 \ar{l}[swap]{p}\ar[hook]{r}{j}	& \widetilde{S}	\ar{r}{\beta}	& S 
	\end{tikzcd}
	$$
Note it is an important fact that in the above situation $X_Z \subset X$ is cut out locally by the same section as $Z \subset S$ and the normal bundle of $X_Z \subset X$ is just the pull-back of the normal bundle $N_i$ of $Z \subset S$. Otherwise the right most square of above diagram for the blowing-ups is only a commutative square, rather than a fibered square.
\begin{definition}\label{def:blcat}
Let $\shA \subset D(X)$ be a $S$-linear admissible subcategories, denote 
	$$\widetilde{\shA} = \shA \boxtimes_S D(\widetilde{S}) \subset D(\widetilde{X}) \quad \text{and} \quad \shA_Z = \shA \boxtimes_S D(Z) \subset D(X_Z)$$
Then $\widetilde{\shA}$ is called the {\em blowing up category of $\shA$ along $\shA_Z$}. 
\end{definition}
Similarly one can define the corresponding categories for perfect complexes:
	$$\widetilde{\shA}^\perf =  \shA \otimes_{\Perf(S)} \Perf(\widetilde{S})\subset \Perf(\widetilde{X}) \quad \text{and} \quad \shA_Z^\perf = \shA \otimes_{\Perf(S)} \Perf(Z) \subset \Perf(X_Z).$$
Then $\widetilde{\shA}^\perf $ is the {\em blowing up category of $\shA$ along $\shA_Z^\perf$}. 
 
From construction and Prop. \ref{prop:product} it follows directly that for any $S$-linear semiorthogonal decomposition $\shA = \langle \shA_1, \shA_2 \rangle$, we have $\widetilde{\shA} = \langle \widetilde{\shA_1}, \widetilde{\shA_2} \rangle$, and $\widetilde{\shA}^\perf = \langle \widetilde{\shA_1}^\perf, \widetilde{\shA_2}^\perf \rangle$. Notice that the projective bundle category $\PP_{\shA_Z}(N_i) \subset D(\PP_{X_Z}(N_i))$ plays the role exceptional divisors of the blowing-up, and is equipped with functors:
	$$p^* \colon \shA_Z \rightleftarrows \PP_{\shA_Z}(N_i) \colon p_*, \qquad j^* \colon \widetilde{\shA} \rightleftarrows   \PP_{\shA_Z}(N_i) \ \colon j_*,$$
induced from the corresponding functors on the derived categories of geometric spaces.

\begin{theorem}[Orlov's blowing up formula \cite{O92}] \label{thm:bl} In the above situation, the $S$-linear functors $\beta^*\colon \shA \to \widetilde{\shA}$ and $\Psi_k = j_* \, p^* (-) \otimes \sO_{\PP(N_i)}(k) \colon \shA_Z  \to  \widetilde{\shA}$ are fully faithful, $k \in \ZZ$. Denote the image of $\shA_Z$ under $\Psi_k$ by $(\shA_Z)_k$, then there are $S$-linear semiorthogonal decompositions
	\begin{align*} \widetilde{\shA}  & = \langle \beta^* \shA, ~ (\shA_Z)_0, (\shA_Z)_1, \ldots, (\shA_Z)_{r-2} \rangle 
				= \langle (\shA_Z)_{1-r}, \ldots, (\shA_Z)_{-2}, (\shA_Z)_{-1}, ~ \beta^* \shA\rangle. 			
	\end{align*}
Similar statements hold for the perfect complexes, in particular, there are $S$-linear semiorthogonal decompositions:
	\begin{align*} \widetilde{\shA}^\perf  & = \langle \beta^* \shA^\perf, ~ (\shA_Z^\perf)_0, \ldots, (\shA_Z^\perf)_{r-2} \rangle 
				= \langle (\shA_Z^\perf)_{1-r}, \ldots, (\shA_Z^\perf)_{-1}, ~ \beta^* \shA^\perf \rangle, 			
	\end{align*}
\end{theorem}

\begin{proof} The blowing up formula for {\em smooth} centers of of Orlov's \cite{O92} holds for $\widetilde{S} \to S$, since $Z \subset S$ is smooth, namely there are $S$-linear semiorthogonal decompositions:
	\begin{align*} D(\widetilde{S})	& = \langle \beta^* D(S), ~ j_*\, p^*\, D(Z), j_* (p^* \, D(Z) \otimes \sO(1)), \ldots, j_* (p^* D(Z) \otimes \sO(r-2)) \rangle;\\
				& = \langle j_* (p^*\,D(Z) \otimes \sO(1-r)), \ldots, j_* (p^*\,D(Z) \otimes \sO(-1)), ~ \beta^* D(S) \rangle,
	\end{align*}
Then we apply base-change along $X \to S$, by Prop. \ref{prop:bcFM} and notice $X \to S$ are fully faithful for the pairs $(S,\widetilde{S})$ and $(Z, \widetilde{S})$ (i.e. for the $S$-schemes $S$, $\widetilde{S}$, $Z$ and $Z \times_{S} \widetilde{S} = E_Z$) by the above Tor-independent squares, we obtain $X$-linear semiorthogonal decompositions
	\begin{align*} D(\widetilde{X})	& = \langle \beta^* D(X), ~ j_*\, p^*\, D(X_Z), j_* (p^* \, D(X_Z) \otimes \sO(1)), \ldots, j_* (p^* D(X_Z) \otimes \sO(r-2)) \rangle;\\
				& = \langle j_* (p^*\,D(X_Z) \otimes \sO(1-r)), \ldots, j_* (p^*\,D(X_Z) \otimes \sO(-1)), ~ \beta^* D(X) \rangle,
	\end{align*}
Notice for the $S$-linear subcategory $\shA$, we have the following diagrams of $S$-linear functors between $S$-linear categories:
	\begin{equation*}
	\begin{tikzcd}[row sep= 2.6 em, column sep = 2.6 em]
	 \PP_{\shA_Z}(N_i) \ar[shift left]{d}{p_*} \ar[shift left]{r}{j_*} & \widetilde{\shA} \ar[shift left]{d}{\beta_*}   \ar[shift left]{l}{j^*}    \\
	\shA_Z  \ar[shift left]{u}{p^*}  \ar[shift left]{r}{i_*}      & \shA  \ar[shift left]{l}{i^*} \ar[shift left]{u}{\beta^*}    
	\end{tikzcd}	
	\end{equation*}
which is a commutative diagram for all the push-forwards and a commutative diagram for all the pull-backs. All the functors $i_*,i^*, j_*,j^*, p_*, p^*, \beta_*, \beta^*$ in the diagram for $\shA$ are induced (and compatible with) the corresponding functors for $X$ restricted to subcategories $\shA \subset D(X)$, $\shA_Z \subset D(X_Z)$, $\PP_{\shA}(N_i) \subset D(E_{X_Z}) = D(\PP_{X_Z}(N_i))$, $\widetilde{\shA}  \subset D(\widetilde{X} )$, and for abuse of notations we still use the same notation. Now the theorem follows from applying Prop. \ref{prop:bcsod} to the $S$-linear category $\shA \subset D(X)$ for the fully faithful base-change to $X \to S$. The results for perfect complexes are similar.
\end{proof}

Even in the case of schemes, the above theorem generalizes Orlov's blowing-up formula to blowing up $\widetilde{X} \to X$ along possibly non-smooth center $Y \subset X$:

\begin{corollary} \label{cor:bl:lci} Suppose $Y$ is a codimension $r \ge 2$ local complete intersection subscheme of a smooth scheme $X$ over a field of characteristic zero, denote $\beta \colon \Bl_Y X \to X$ the blowing up of $X$ along $Y$, and $E$ the exceptional divisor. Then the (derived) functors $\beta^*\colon D(X)\to D(\Bl_Y X)$ and $\Psi_k = j_* \, p^* (-) \otimes \sO(-kE) \colon D(Y) \to D(\Bl_Y X)$ are fully faithful, $k \in \ZZ$. Denote the image of $\Psi_k$ by $D(Y)_k$. Then there are $X$-linear semiorthogonal decompositions
	\begin{align*} D(\Bl_Y X)	& = \langle \beta^* D(X), ~ D(Y)_0, D(Y)_1, \ldots, D(Y)_{r-2} \rangle;\\
				& = \langle D(Y)_{1-r}, \ldots, D(Y)_{-2}, D(Y)_{-1}, ~ \beta^* D(X) \rangle.
	\end{align*}
Similarly for the categories of perfect complexes.
\end{corollary}
\begin{proof} By passing to smaller open subschemes, we may assume $Y$ is locally cut out by a regular sequence $f_1, \ldots, f_r \in \Gamma(\sO_X)$. Therefore locally there is a morphism $X \to S = \AA^n = \Spec \ZZ[z_1, \ldots, z_n]$, $n > r$, and $Z= \AA^{n-r} \subset \AA^n$ is the subscheme $z_1 = \cdots = z_r = 0$, such that $Y = X \times_S Z$. Apply previous theorem to $\shA = D(X)$, $\shA_Z = D(Y)$.
\end{proof}

\begin{remark} \label{rmk:thm:bl:nreg} Note that the construction of this section, the proof Thm. \ref{thm:bl} (therefore Cor. \ref{cor:bl:lci}) works in a more general situation without smoothness assumption on $X$. Assume $S \to \overline{S}$ is a morphism of schemes such that $Z \subset S$ is the zero locus of a regular section of the pulling back of a vector bundle $\sE$ on $\overline{S}$, therefore the blowing up diagram (\ref{diag:Bl_Z S}) is $\overline{S}$-linear. Let $\overline{X} \to \overline{S}$ be a morphism of schemes such that $\overline{X}$ is smooth and the fiber square for $X : = \overline{X} \times_{\overline{S}} S$ is Tor-independent, and assume further that $X_Z : = Z \times_S X \subset X$ is also cut out by the pulling-back of the same section of the vector bundle $\sE$. Then $\widetilde{X} = X \times_S \widetilde{S}$ is the blowing up of $X$ along $X_Z$, and the corresponding blowing up diagram is $\overline{X}$-linear. Let $\overline{\shA} \subset D(\overline{X})$ be a $\overline{S}$-linear admissible subcategory, then $\shA: = \overline{\shA}  \boxtimes_{\overline{S}} D(S)$ is a $S$-linear admissible subcategory of $D(X)$ (with projection functor of finite homological amplitude). Then $\widetilde{\shA} = \shA \boxtimes_{D(S)} D(\widetilde{S}) =  \overline{\shA}  \boxtimes_{\overline{S}} D(\overline{X})$ is the {\em blowing up category} of $\shA$ along $\shA_Z = \shA \boxtimes_{D(S)} D(Z)$, and Thm. \ref{thm:bl} holds for $\widetilde{\shA}$ by applying fully faithful base-change $\overline{X} \to \overline{S}$.
\end{remark}

The following will be useful later.
\begin{lemma} \label{lem:mut:bl} In the situation of Thm. \ref{thm:bl}, the right mutation functor
satisfies the following:
	$$\RR_{(\shA_Z)_k} (\beta^* A \otimes \sO_{\widetilde{X}}(-kE)) = \beta^* A \otimes  \sO_{\widetilde{X}}(-(k+1)E), \qquad A \in \shA, \,k \in \ZZ.$$
\end{lemma}

\begin{proof} The similar statement for left mutations is proved by Carocci-Turcinovic {\cite[Prop. 3.4]{CT15}}. 
For the sake of completeness, we include a proof here. For simplicity of notations we denote exceptional divisor by $E$, and denote $\sO(kE) = \sO_{\widetilde{X}}(kE)$, for $k \in \ZZ$. Note that $\sO(kE)|_{E} = \sO_{\PP(N_i)}(-k)$. Since the mutation functor satisfies $\sigma \circ \RR_{\shB} = \RR_{\sigma(\shB)} \circ \sigma$ for any admissible subcategory $\shB \subset \shT$ and any autoequivalence $\sigma: \shT \to \shT$, therefore we only need to show the case for $k=-1$, i.e. to show 
	$$\RR_{(\shA_Z)_{-1}} (\beta^* A \otimes \sO(E)) = \beta^* A, \qquad  \forall ~A \in \shA,$$ 
then other cases will follow from applying $\sigma \circ \RR_{\shB} = \RR_{\sigma(\shB)} \circ \sigma$ to the autoequivalence $\sigma = \otimes \sO(kE)$. Denote $\shB =p^* (\shA_Z) \otimes \sO_E(-1) \subset \PP_{\shA}(N_i)$, then $j_* \shB =  (\shA_Z) _{-1}$. Denote $i_{j_* \shB}$ the inclusion $j_* \shB  \subset  \widetilde{X}$. We need to compare $\RR_{j_* \shB}$ with $j_* \, \RR_{\shB}\, j^*$. For any $A \in \shA$, consider the commutative diagram:
	\begin{equation} \label{eqn:oct}
	\begin{tikzcd}
	 j_* \,  (i_{\shB}\, i_{\shB}^*)\, j^*  & i_{j_* \shB}\, i_{j_* \shB}^*  \ar{l}[swap]{\sim} & 0 \ar{l}   \\
	 j_* j^*   \ar{u} & \id \ar{l}  \ar{u} & \otimes \sO(-E) \ar[equal]{d} \ar{u}  \ar{l} \\
	 j_* \, \RR_{\shB}\, j^*  \ar{u} & \RR_{j_* \shB} \ar[dashed]{l} \ar{u} & \otimes  \sO(-E) \ar[dashed]{l}.
	\end{tikzcd}
	\end{equation}
From octahedral axiom the last row is an exact triangle. Since $j^* (\beta^*A \otimes \sO(E)) = p^*A \otimes \sO_E(-1) \in \shB$, therefore $j_* \RR_{\shB} j^* (\beta^*A \otimes \sO(E)) = j_* \RR_{\shB} (p^*A \otimes \sO_E(-1)) = 0$. Hence $\RR_{j_* \, \shB} (\beta^*A \otimes \sO(E))=  (\beta^*A \otimes \sO(E))\otimes \sO(-E) = \beta^*A$.\end{proof}

\begin{remark} Similarly one can prove 	
	$$  \LL_{(\shA_Z)_k} (\beta^* A  \otimes \sO_{\widetilde{X}}(-(k+1)E)) = \beta^* A \otimes \sO_{\widetilde{X}}(-kE), \qquad A \in \shA, \,k \in \ZZ.
	$$
\end{remark}

\subsection{Generalized universal hyperplane} \label{sec:hyp} Let $S$ be a smooth $B$-scheme, $i \colon Z \hookrightarrow S$ be a smooth subscheme. We further assume that $Z = Z(s)$ is the zero locus of a regular section $s \in \Gamma(S,E)$ of a vector bundle $E$ of rank $r$. Recall the generalized universal hyperplane $\shH_s \subset \PP_S(E^\vee)$ is the hypersurface cut out by the section of $\sO_{\PP_S(E^\vee)}(1)$ under the identification 
	$$\Gamma(\PP_S(E^\vee), \sO_{\PP(E^\vee)}(1)) = \Gamma(S, E).$$
Denote $\pi \colon \shH_s \to S$ the projection, then $\pi$ is a $\PP^{r-2}$-bundle  over $S\, \backslash \,Z$, and $\shH_s|_{\pi^{-1}(Z)} = \PP_Z(E^\vee|_Z) = \PP_Z(N_i^\vee)$, where $N_i$ is the normal bundle of $Z \subset S$ as usual. 
Therefore we have a commutative diagram:
	\begin{equation*}
	\begin{tikzcd}[row sep= 2.6 em, column sep = 2.6 em]
	\PP_Z(N_i^\vee) \ar{d}[swap]{\rho} \ar[hook]{r}{j} & \shH_s \ar{d}{\pi} \ar[hook]{r}{\iota} & \PP_S(E) \ar{ld}[near start]{q} 
	\\
	Z \ar[hook]{r}{i}         & S
	\end{tikzcd}	
	\end{equation*}

Let $a_X \colon X \to S $ be a smooth proper $S$-scheme such that $X_Z := X \times_S Z$ is of expected dimension $\dim X -r$. Then $X_Z$ is also cut out by the section $a_X^{*}\,s \in H^0(X, a_X^*E)$. Therefore we can similarly form the generalized universal hyperplane $\shH_{X,s} \subset \PP_X(E)$ for $X$ with respect to the bundle $a_X^*E$ and section $a_X^*\,s$. By abuse of notation we will denote the bundle $a_X^* E$ and section $a_X^*\,s$ on $X$ still by $E$ and $s$ respectively, and denote the maps by same notations, i.e. we denote the inclusions by $i\colon X_Z \hookrightarrow X$ and $j \colon \PP_{X_Z}(N_i^\vee) \hookrightarrow \shH_{X,s}$, and the projections by $\rho\colon \PP_{X_Z}(N_i^\vee) \to X_Z$ and $\pi\colon \shH_{X,s} \to X$.


\begin{definition} Assume $\shA \subset D(X)$ is an admissible $S$-linear subcategory, then denote 
	$$\shH_{\shA,s} = \shA \boxtimes_S D(\shH_s ) \subset D(\shH_{X,s}), \qquad \shH_{\shA,s}^\perf =  \shA \otimes_{\Perf(S)} \Perf(\shH_s) \subset \Perf(\shH_{X,s}).$$
Then $\shH_{\shA,s}$ (resp. $\shH_{\shA,s}^\perf$) is called the {\em generalized universal hyperplane} for $\shA$ with respect to vector bundle $E$ and regular section $s$.
\end{definition}

As blowing up case, various adjoint functors on derived categories of schemes induce adjoint functors on corresponding subcategories, and we have the following diagrams of $S$-linear functors:
\begin{equation*}
	\begin{tikzcd}[row sep= 3 em, column sep = 3 em]
	 \PP_{\shA_Z}(N_i^\vee) \ar[shift left]{d}{\rho_*} \ar[shift left]{r}{j_*} & \shH_{\shA,s}  \ar[shift left]{d}{\pi_*}   \ar[shift left]{l}{j^*}  \ar[shift left]{r}{\iota_*} & \PP_\shA (E^\vee) \ar[shift left]{ld}{q_*}   \ar[shift left]{l}{\iota^*} 
	\\
	\shA_Z  \ar[shift left]{u}{\rho^*}  \ar[shift left]{r}{i_*}      & \shA  \ar[shift left]{l}{i^*} \ar[shift left]{u}{\pi^*}   \ar[shift left]{ru}{q^*}
	\end{tikzcd}	
\end{equation*}
which are commutative for all push-forwards and respectively for all pull-backs. All these functors are induced by and compatible with the corresponding functors for the ambient schemes under the inclusions $\shA \subset D(X)$, $\shA_Z \subset D(X_Z)$, $\PP_{\shA_Z}(N_i^\vee) \subset D(\PP_{X_Z}(N_i^\vee))$, $\PP_{\shA}(E^\vee) \subset D(\PP_{X}(E^\vee))$ and $\shH_{\shA,s}  \subset D(\shH_{X,s})$, as usual. Similarly for the categories of perfect complexes.

\begin{theorem}[Orlov's generalized hyperplane theorem, {\cite[Prop. 2.10]{O}}] \label{thm:hyp} In the above situation, the functors $j_*\,\rho^* \colon \shA_Z  \to \shH_{\shA,s}$ and $ \pi^*(-) \otimes \sO_{\PP(E^\vee)}(k) \colon \shA  \to \shH_{\shA,s} $ are fully faithful, $k \in \ZZ$, and there is $S$-linear semiorthogonal decompositions:
	\begin{align*}
	\shH_{\shA,s} & = \langle  j_* \, \rho^* \shA_Z,  ~~\pi^* \shA \otimes \sO_{\PP(E^\vee)}(1), \ldots , \pi^* \shA \otimes  \sO_{\PP(E^\vee)}(r-1)\rangle \\
	& = \langle \pi^* \shA \otimes \sO_{\PP(E^\vee)}(2-r), \ldots,  \pi^* \shA \otimes \sO_{\PP(E^\vee)}, ~~  j_* \, \rho^* \shA_Z \rangle,
	\end{align*}
and similar decompositions for perfect complexes.
\end{theorem}
\begin{proof} The same formula holds for $\shH_s$ over $S$ by Orlov's result {\cite[Prop. 2.10]{O}}. The desired formula follows from base-change along $X \to S$ as we did for blowing up case. 
\end{proof}

\section{HPD with base-locus} \label{sec:app:HPDbs}

Let $X$ be a smooth $B$-scheme with map $X \to \PP(V)$ and $L \subset V^\vee$ be a linear subbundle of rank $\ell$ over $B$. We denote by $L^\perp = \Ker (V \to L^\vee) \subset V$ the orthogonal bundle as usual.  Then $L$ is a linear system through the composition $L \to V^\vee \to H^0(X,\sO_{\PP(V)}(1))$, which determines a rational map  $X \dashrightarrow \PP(L^\vee)$, with base locus $X_{L^\perp}: = X \times_{\PP(V)} \PP(L^\perp)$. We assume $X_{L^\perp}$ is of expected dimension $\dim X - \ell$. One can resolve the indeterminacy of  $X \dashrightarrow \PP(L^\vee)$ blowing up $X$ along base-locus $X_{L^\perp}$. Then by construction there is a natural map $\widetilde{X} = \Bl_{X_{L^\perp}} X \to \widetilde{\PP(V)} :=  \Bl_{\PP(L^\perp)} \PP(V) \to \PP(L^\vee)$ which makes 
$\widetilde{X}$ a $\PP(L^\vee)$-scheme. This map is also compatible with the composition $\widetilde{X} \hookrightarrow X \times \PP(L^\vee) \to \PP(L^\vee)$, where the last map is the projection to second factor. The situation (as in \S \ref{sec:bl}) is summarized in the following diagram, with names of maps as indicated:
	\begin{equation}\label{eqn:app:bl:bsp}
	\begin{tikzcd}
	X_{L^\perp} \times \PP(L^\vee) \ar{d}[swap]{p} \ar[hook]{r}{j} & \widetilde{X} \ar{d}{\beta} \ar[hook]{r}{\iota} & X \times \PP(L^\vee) \ar{ld}[near start]{q} 
	\\
	X_{L^\perp} \ar[hook]{r}{i_L}         & X
	\end{tikzcd}	
	\end{equation}
	
Let $\shA \subset D(X)$ be a $S=\PP(V)$-linear Lefschetz subcategory with Lefschetz center $\shA_0$, Lefschetz components $\shA_k$ and length $m$. We can apply the construction of \S \ref{sec:bl} to $\shA$, and consider the blowing up category 
	$$\widetilde{\shA} := \shA \boxtimes_{\PP(V)} D(\widetilde{\PP(V)}) \subset D(\widetilde{X})$$
of $\shA$ along $\shA_{\PP(L^\perp)}$, where 
	$$\shA_{\PP(L^\perp)} = \shA \boxtimes_{\PP(V)} D(\PP(L^\perp))$$ 
can be regarded as the {\em base-locus category} of $\shA$ for the linear system $L \subset V^\vee$. Then blowing-up category $\widetilde{\shA}$ is equipped with a $\PP(L^\vee)$-linear structure from the projection $\widetilde{\PP(V)} \to \PP(L^\vee)$. 

\subsection{Refined blowing up $\rBl_{\shA_{\PP(L^\perp)}} \shA$ and Lefschetz structure}

Notice that the $\PP(L^\vee)$-linear category $\widetilde{\shA}$ admits a semiorthogonal decomposition by Thm. \ref{thm:bl}:
	\begin{align}  \widetilde{\shA} & = \big\langle \beta^*\shA ,~~ (\shA_{L^\perp})_0, \ldots, (\shA_{L^\perp})_{\ell-2} \big\rangle \notag \\
	& = \big\langle  \beta^*\shA_0, \beta^* (\shA_1(1)), \ldots, \beta^* (\shA_{m-1}(m-1)), ~ ~(\shA_{L^\perp})_0, \ldots, (\shA_{L^\perp})_{\ell-2} \big\rangle.  \label{sod:app:bl-bs} 
	\end{align}

\begin{lemma} \label{lem:app:ref-bl} The action functor $\act \colon \widetilde{\shA} \boxtimes D(\PP(L^\vee)) \to \widetilde{\shA}$ is fully faithful on the following subcategories of $\widetilde{\shA} \boxtimes D(\PP(L^\vee))$:
	$$\beta^*(\shA_{\ell-1})\boxtimes D(\PP(L^\vee)), \beta^*(\shA_{\ell}(1))\boxtimes D(\PP(L^\vee)) \ldots, \beta^*(\shA_{m-1}(m-\ell))\boxtimes D(\PP(L^\vee)),$$
and their images remain a semiorthogonal sequence in $\widetilde{\shA}$. 
\end{lemma}

Note that this result is only interesting if $m \ge \ell$. 

\begin{proof}[Proof of Lem. \ref{lem:app:ref-bl}] Denote $\widetilde{\Gamma} \colon \widetilde{X} \hookrightarrow \widetilde{X}  \times \PP(L^\vee)$ the graph embedding of the morphism $ \widetilde{X} \to \PP(L^\vee)$, then one has a commutative diagram:
	\begin{equation} \label{eqn:app:bl-graph}
	\begin{tikzcd} & \widetilde{X}  \times \PP(L^\vee) \ar{d}{\beta \times \Id} \\
	\widetilde{X} \ar[hook]{ur}{\widetilde{\Gamma}} \ar[hook]{r}{\iota}& X\times   \PP(L^\vee)
	\end{tikzcd}
	\end{equation}
The normal bundle of $\widetilde{\Gamma}$ is $N_{\widetilde{\Gamma}} = \sL \boxtimes \shT_{\PP(L^\vee)}(-1)$, where $\sL$ is the line bundle $\sO_{\PP(L^\vee)}(1)|_{\widetilde{X}} = \beta^*\sO_X(1) \otimes \sO_{\widetilde{X}}(-E)$. Then $\widetilde{\Gamma}^* = \act \colon \widetilde{\shA} \boxtimes D(\PP(L^\vee)) \to \widetilde{\shA}$ is the action functor. To prove the lemma, we  need to show the vanishing of the following $\Hom$ space: 
	$$\Hom(\act(A_1 \boxtimes F_1), \act(A_2 \boxtimes F_2))$$
for any $A_1,A_2 \in \widetilde{\shA}, F_1, F_2 \in D(\PP(L^\vee))$. Similar to the proof of  Lem. \ref{lem:non-moderate Lef}, from Koszul resolution for $\widetilde{\Gamma}_* \sO_{\widetilde{X}}$, the above $\Hom$ space is an iterated cone of the $\Hom$ spaces:
	\begin{align} \label{eqn:app:cone:act}
	  &\Hom_{\widetilde{\shA}}\big(A_1,A_2 \otimes \sL^{-r} \big) \otimes \Hom_{\PP(L^\vee)}\big(F_1, F_2 \otimes \Omega^r(r) \big)[-r], &0\le r \le \ell-1,
	\end{align}
where the case $r=0$ corresponds to $\Hom(A_1 \boxtimes F_1, A_2 \boxtimes F_2)$. Then the desired vanishing will follow from: for any $A_k \in \beta^*(\shA_{i_k}(i_k +1 - \ell))$, $k=1,2$, such that $\ell -1 \le i_2 \le i_1 \le m-1$,
	$$\Hom_{\widetilde{\shA}}\big(A_1,A_2 \otimes \sL^{-r} \big) =0, \qquad \forall 1 \le r \le \ell-1.$$
In fact, for all $1 \le r \le \ell-1$, if we right mutate $\beta^*(\shA_{i_1}(i_1))$ passing through $(\shA_{L^\perp})_0, \ldots, (\shA_{L^\perp})_{r-1}$ of Thm. \ref{thm:bl}, it becomes $\beta^*(\shA_{i_1}(i_1-r)) \otimes \sL^r$ by Lem. \ref{lem:mut:bl}. Since $0 \le i_2 - r \le i_1 -1$, hence 
	$$ \beta^* (\shA_{i_2-r}(i_2-r)), \beta^*(\shA_{i_1}(i_1-r)) \otimes \sL^r  $$
is a semiorthogonal sequence of $\widetilde{\shA}$.
The desired result follows from $\shA_{i_2} \subseteq \shA_{i_2-r} \subseteq \shA_{0}$.
\end{proof}

\begin{remark} From the commutative diagram (\ref{eqn:app:bl-graph}), the lemma is equivalent to that the functor $\iota^* \colon \shA \boxtimes D(\PP(L^\vee)) \to  \widetilde{\shA}$ is fully faithful on the subcategories:
	$$\shA_{\ell-1} \boxtimes D(\PP(L^\vee)), \ldots,\shA_{m-1}(m-\ell) \boxtimes D(\PP(L^\vee)),$$
and their images form a semiorthogonal sequence in $\widetilde{\shA}$. Therefore the lemma can also be proved by using Koszul complex $N_{\iota} ^\bullet$ for the embedding $\iota$, where $N_{\iota} = \sO_X(1) \boxtimes \shT_{\PP(L^\vee)}(-1)$.
\end{remark}

Denote the subcategory generated by the images of above lemma by:
	\begin{align*} \widetilde{\shA}^{\rm amb}: & = \big \langle  \act \big(\beta^*(\shA_{\ell-1}) \boxtimes D(\PP(L^\vee))\big), \ldots, \act\big(\beta^*(\shA_{m-1}(m-\ell)) \boxtimes D(\PP(L^\vee))\big) \big \rangle \\
	&  = \big \langle \iota^*(\shA_{\ell-1} \boxtimes D(\PP(L^\vee))), \ldots, \iota^*((\shA_{m-1}(m-\ell) \boxtimes D(\PP(L^\vee))) \big \rangle \subset \widetilde{\shA}.
	\end{align*}
Then $\widetilde{\shA}^{\rm amb}$ is the ``trivial piece" of the $\PP(L^\vee)$-linear structure of $\widetilde{\shA}$, and its orthogonal captures the essential $\PP(L^\vee)$-linear information of the blowing up $\widetilde{\shA}$. Note that $\widetilde{\shA}^{\rm amb} \ne \emptyset$ if and only if $m \ge \ell$. This leads to the concept of a {\em refined} blowing-up:

\begin{definition} For a $\PP(V)$-linear Lefschetz category $\shA$ of length $m$, a subbundle $L \subset V^\vee$ of rank $\ell$ as above. Then {\em refined blowing-up category of $\shA$ along $\shA_{\PP(L^\perp)}$}, denoted by $\rBl_{\shA_{\PP(L^\perp)}} \shA$ (or simply by $\rBlA$ if there is no confusion), is defined to be the $\PP(L^\vee)$-linear subcategory which is right orthogonal to the image of Lem. \ref{lem:app:ref-bl}, i.e.
	$$ \rBlA \equiv \rBl_{\shA_{\PP(L^\perp)}} \shA : = (\widetilde{\shA}^{\rm amb})^\perp \subset \widetilde{\shA}.$$
In particular by definition, $\rBlA = \widetilde{\shA}$ if $m < \ell$, and $\rBlA \subsetneqq \widetilde{\shA}$ if $m \ge \ell$.
\end{definition}

\begin{proposition}\label{prop:app:lef:ref-bl} The refined blowing-up category $\rBlA$ admits a (moderate) $\PP(L^\vee)$-linear Lefschetz structure, with Lefschetz center
	$\rBlA_0$ and a (right) Lefschetz decomposition
	$$ \rBlA= \langle \rBlA_0, \rBlA_1 \otimes \sL, \ldots, \rBlA_{\ell-2} \otimes \sL^{\otimes \ell -2} \rangle,$$
where $\sL = \sO_{\PP(L^\vee)}(1)|_{\widetilde{X}} = \beta^*\sO_X(1) \otimes \sO_{\widetilde{X}}(-E)$ as before, and 
	$$\rBlA_k  :=\langle \beta^* \shA_k', (\shC_L)_0 \rangle, \qquad 0 \le k \le \ell-2.$$
Here $\shA_k'$ for $k \ge 0$ is the subcategory of $\shA_k$ defined by 
	$$\shA_k' := \shA_{\ell-1}^\perp \cap \shA_k \subset \shA_0,$$
and $\shC_L$ is the essential component of $\shA_{L^\perp}$ defined by
	\begin{align*} 
	\shC_L &:= \big\langle i_L^* \shA_{\ell}(1), \ldots, i_L^* \shA_{m-1}(m-\ell) \big \rangle^\perp \subset \shA_{L^\perp} \\
	& = \{ C \in \shA_{L^\perp} \mid i_{L*} C \in \langle \shA_0(1-\ell), \shA_{1}(2-\ell) \ldots, \shA_{\ell-1} \rangle \subset \shA\},
	\end{align*}
where $i_L \colon X_{L^\perp} \to X$ is the inclusion. Finally $(\shC_L)_0$ denotes  $j_*\, p^*(\shC_L)$ as in Thm. \ref{thm:bl}.
\end{proposition}

The fact that $i_L^*$ is fully faithful on $\shA_{\ell}(1), \ldots, \shA_{m-1}(m-\ell)$ follows from the Koszul resolution for $X_{L^\perp} \subset X$. Therefore from adjunction, there is a semiorthogonal decomposition
	$$\shA_{L^\perp} = \langle \shC_L, ~~i_L^* \shA_{\ell}(1), \ldots, i_L^* \shA_{m-1}(m-\ell)\rangle.$$
The nontrivial component $\shC_L$ of $\shA_{L^\perp}$, called {\em Kuznetsov component} of $\shA_{L^\perp}$, plays an essential role in Kuznetsov's HPD theory \cite{Kuz07HPD}. In the terminology of \cite{JLX17}, $\shC_L= {}^{\rm prim} \shA_{L^\perp}$ is the left primitive component of $\shA_{L^\perp}$.

\begin{proof} Applying Lem. \ref{lem:mut:bl} to (\ref{sod:app:bl-bs}), we obtain that the following categories
	$$\langle \beta^* \shA_0', (\shC_L)_0 \rangle, \langle \beta^* \shA_1' \otimes \sL, (\shC_L)_0 \otimes \sL \rangle \ldots,  \langle \beta^*\shA_{\ell-2}' \otimes \sL^{\ell-2}, (\shC_L)_{0} \otimes \sL^{\ell-2} \rangle,$$
forms a semiorthogonal sequence of $\widetilde{\shA}$. Denote their span by $\shB$, then we have $\shB = \langle \rBlA_0, \ldots, \rBlA_{\ell-e} \otimes \sL^{\otimes \ell -e} \rangle$, by the definition of $\rBlA_k$. Then to show the proposition is equivalent to show that $\shB = \rBlA$ as subcategories of $\widetilde{\shA}$.


\medskip \noindent{\em Observation.} From the decomposition (\ref{sod:app:bl-bs}), if we right mutation $\beta^*(\shA_k(k))$ passing through $(\shA_{L^\perp})_0, \ldots, (\shA_{L^\perp})_{r-1}$, where $k \ge 0$ and $1 \le r \le \ell-1$, we end up with $\beta^*(\shA_k(k-r)) \otimes \sL^r$ by Lem. \ref{lem:mut:bl} (which is also true for $r=0$). Therefore we obtain the following:
	\begin{equation}\label{eqn:app:van:rbl}
	\Hom\big(\beta^*(\shA_k(k-r)) \otimes \sL^r, \beta^*(\shA_j(j)) \big) = 0, \qquad 0 \le j < k,~ 0 \le r \le \ell-1,
	\end{equation}
as well as for $s=0,1,\ldots, r-1$:
	\begin{equation}\label{eqn:app:van:rbl-bs}
	\Hom\big(\beta^*(\shA_k(k-r) \otimes \sL^{r}, (\shA_{\PP(L^\perp)})_{0} \otimes \sL^s \big) = 0, \qquad k\ge 0, ~ 0 \le s < r \le \ell-1,
	\end{equation}
since $(\shA_{\PP(L^\perp)})_s = (\shA_{\PP(L^\perp)})_{0} \otimes \sL^s$, where $\sL = \beta^*\sO_X(1) \otimes \sO(-E)$. 

\medskip \noindent{\em Step 1.}  We show that the following components of $\shB$ are contained in $\rBlA$:
	$$\big\langle \beta^* \shA_0', \beta^* \shA_1' \otimes \sL, \ldots, \beta^*\shA_{\ell-2}' \otimes \sL^{\ell-2}\big \rangle \subset  (\widetilde{\shA}^{\rm amb})^\perp = \rBlA .$$

This is equivalent to show, for any $A_1 \in \shA_{i_1}(i_1+1-\ell)$, $F_1 \in D(\PP(L^\vee))$, and $A_2 \in \shA_{i_2}'$ with $0 \le i_2 < \ell-1 \le i_1 \le m-1$, the following holds: 
	$$\Hom(\act(\beta^*A_1 \boxtimes F_1), \beta^* A_2 \otimes \sL^{i_2}) = \Hom(\act(A_1 \boxtimes F_1), \act(A_2 \boxtimes \sO_{\PP(L^\vee)}(i_2))=0$$
Similar to the proof of Lem. \ref{lem:app:ref-bl}, the above $\Hom$ space is an iterated cone of the $\Hom$ space of (\ref{eqn:app:cone:act}). Therefore we only need to show
	$$\Hom_{\widetilde{\shA}}(\beta^*A_1 \otimes \sL^r, \beta^* A_2) =0, \qquad 0 \le r \le \ell-1,$$
for all $A_1 \in \shA_{i_1}(i_1+1-\ell)$, $A_2 \in \shA_{i_2}'$, where $i_1, i_2$ are integers that  $0 \le i_2 < \ell-1 \le i_1 \le m-1$.

First consider the case $i_1 = \ell -1$. For $r=0$, $\Hom(\beta^* A_1, \beta^*A_2) = \Hom(A_1, A_2) = 0$, since $A_1 \in \shA_{\ell-1}$, and $A_2 \in \shA_{i_2}' \subset \shA_{\ell-1}^\perp$, by the way we define $\shA_{k}'$. 
For $1 \le r \le \ell-1$, then the vanishing follows from (\ref{eqn:app:van:rbl}) applied to $k=r > j=0$, and $\shA_{\ell-1} \subset \shA_r$, $\shA_{i_2}' \subset \shA_0$.

For $i_1 = \ell, \ldots, m-1$, then the vanishing $\Hom(\beta^*A_1 \otimes \sL^r, \beta^* A_2) =0$ for $0 \le r \le \ell-1$ follows from applying (\ref{eqn:app:van:rbl}) to $k = i_1 + r - (\ell-1) \in [i_1-(\ell-1), i_1] \subset [1,m-1]$ and $j = 0 < k$, and that $\shA_{i_1} \subset \shA_{k}$ hence $\shA_{i_1}(i_1 + 1 -\ell) \subset \shA_{k}(k-r)$, and $\shA_{i_2}' \subset \shA_0$.

\medskip \noindent{\em Step 2.} The remaining components of $\shB$ (for Step 1) are also contained in $\rBlA$, i.e.
	$$\big\langle (\shC_L)_0, (\shC_L)_{0} \otimes \sL, \ldots, (\shC_L)_{0} \otimes \sL^{\ell-2}  \big\rangle \subset (\widetilde{\shA}^{\rm amb})^\perp = \rBlA.$$
This is equivalent to show
	$$\Hom_{\widetilde{\shA}}(\act(\beta^*A \boxtimes F), ~ j_* p^* C\otimes \sL^r) =0, \qquad 0 \le r \le \ell-2,$$
for all $A \in \shA_{i}(i+1-\ell)$ where $\ell -1 \le i \le m-1$, $F \in D(\PP(L^\vee))$, $C \in \shC_L$. 

Notice that for $i = \ell-1$, this follows directly from (\ref{eqn:app:van:rbl-bs}) applied to $k=r=\ell-1$. Now we focus on the case $i = \ell, \ldots, m-1$. From commutative diagram (\ref{eqn:app:bl-graph}), 
	$$\act(\beta^*A \boxtimes F)=\Gamma^*(\beta^* A \boxtimes F) = \iota^*(A \boxtimes F).$$
On the other hand, since the ambient square of (\ref{eqn:app:bl:bsp}) is Tor-independent (as it is a flat base-change), 
	$$\iota_* (j_* p^* C\otimes \sL^r) = q^* i_{L*} C \otimes \sO_{\PP(L^\vee)}(r) =  i_{L*} C \boxtimes \sO_{\PP(L^\vee)}(r).$$
Hence by adjunction, 
	\begin{align*}
	&\Hom_{\widetilde{\shA}}(\act(\beta^*A \boxtimes F), ~ j_* p^* C\otimes \sL^r) = \Hom(\iota^*(A \boxtimes F), ~ j_* p^* C\otimes \sL^r)  \\
	=&  \Hom_{\shA \boxtimes D(\PP(L^\vee))}(A \boxtimes F,~ \iota_* (j_* p^* C\otimes \sL^r))  =  \Hom(A \boxtimes F, ~  i_{L*} C \boxtimes \sO_{\PP(L^\vee)}(r))  \\
	=& \Hom_{\shA}(A, i_{L*} C) \otimes \Hom_{\PP(L^\vee)}(F, \sO_{\PP(L^\vee)}(r))  = 0.
	\end{align*}
Last equality follows from $\Hom_{\shA}(A, i_{L*} C) = \Hom_{\shA_{L^\perp}}(i_L^* A, C)= 0$ for $A \in \shA_{i}(i+1-\ell)$, $i=\ell, \ldots, m-1$, $C \in \shC_L$. This is precisely the definition of $\shC_L$.

\medskip \noindent{\em Generation.} From last two steps, we have shown that $\shB \subset \rBlA$. It remains to show $\shB$ generates $\rBlA$, or equivalently $\shB$ together with $\widetilde{\shA}^{\rm amb}$ generate $\widetilde{\shA}$. Compare the span of $\shB$ and $\widetilde{\shA}^{\rm amb}$ with the semiorthogonal decomposition (\ref{sod:app:bl-bs}) of $\widetilde{\shA}$, we only need to show
	$$(i_L^*\shA_\ell(1))_0 \otimes \sL^r, \ldots, (i_L^* \shA_{m-1} (m-\ell))_0 \otimes \sL^r$$
belongs to the category generated by $\shB$ and $\widetilde{\shA}^{\rm amb}$ for all $r=0,\ldots,\ell-2$. The category $(i_L^*\shA_i(i+1-\ell))_0 \otimes \sL^r$ for  $\ell \le i \le m-1$ and $0 \le r \le \ell-2$ is generated by elements 
	$$ j_* p^* (i_L^*A) \otimes \sL^r = j_* j^*\iota^* (A \boxtimes \sO_{\PP(L^\vee)}(r)), \qquad A \in \shA_i(i+1-\ell).$$
From the distinguished triangle of functors:
	$$\otimes \sO(-E) \to \Id \to j_* j^* \xrightarrow{[1]}{},$$
the element $j_* p^* (i_L^*A) \otimes \sL^r$ is isomorphic to the cone:
	\begin{align*}
	j_* p^* (i_L^*A) \otimes \sL^r = &\cone \Big( \iota^* (A \boxtimes \sO_{\PP(L^\vee)}(r)) \otimes \sO(-E) \to \iota^* (A \boxtimes \sO_{\PP(L^\vee)}(r))\Big) \\
	= & \cone  \Big( \act \big(\beta^*(A(-1)) \boxtimes \sO_{\PP(L^\vee)}(r+1)\big) \to \act \big(\beta^*A\boxtimes \sO_{\PP(L^\vee)}(r)\big) \Big).
	\end{align*}
Since $A \in  \shA_i(i+1-\ell)$, $A(-1) \subset \shA_{i}(i-\ell) \subset \shA_{i-1}(i-\ell)$, for $\ell \le i \le m-1$, both terms of the above cone belong to $\widetilde{\shA}^{\rm amb}$. Hence we are done.
\end{proof}

\subsection{HPD with base-locus}

Next we show that the HPD of $\rBlA$ over $\PP(L^\vee)$ is simply given by the linear section $(\shA^\hpd)|_{\PP(L)}$ of $\shA^\hpd$, which can also be intrinsically defined as the nontrivial component $\shD_L$ of the generalized universal hyperplane of $\shA$ over $\PP(L)$. More precisely, apply the construciton of \S \ref{sec:hyp} to the case $S = \PP(V)$, $Z = \PP(L^\perp)$, $E = L^\vee \otimes \sO_{\PP(V)}(1)$, $i \colon \PP(L^\perp) \hookrightarrow \PP(V)$. Consider the section $s_L$ of $E$ which is the canonical section which corresponds to the inclusion $L\subset V^\vee$ under the identification:
 	$$s_L \in \Gamma(\PP(V), L^\vee \otimes \sO_{\PP(V)}(1)) = \Hom_B(L, V^\vee).$$
Therefore we can form the {\em generalized universal hyperplane} for $X \to \PP(V) \dasharrow \PP(L^\vee)$:
	$$\shH_{X,L} := \shH_{X, s_L}  \hookrightarrow \PP_X(E) = X \times \PP(L),$$
where $\delta_{\shH_{L}} \colon \shH_{X,L} \hookrightarrow X \times \PP(L)$ is an inclusion of the divisor $\sO(1,1):=\sO_X(1) \boxtimes \sO_{\PP(L)}(1)$. Note that the inclusion $\delta_{\shH_{L}}$ is also the restriction of the inclusion of universal hyperplane $\delta_{\shH} \colon \shH_X \hookrightarrow X \times \PP(V^\vee)$ to $\PP(L) \subset \PP(V^\vee)$. Using the notation of \S \ref{sec:hyp}, we have $\PP_Z(N_i^\vee) = \PP_{\PP(L^\perp)}(E^\vee)  = \PP(L^\perp) \times \PP(L)$, and $j \colon  \PP(L^\perp) \times \PP(L) \hookrightarrow \shH_{\PP(L^\perp), L}$ is the inclusion of fiber of $\shH_{\PP(L^\perp), L}$ over $Z = \PP(L^\perp)$. By abuse of notation, we also denote the base-change of $j$ to $X \to S=\PP(V)$ by the same notation $j \colon X_{L^\perp} \times \PP(L) \hookrightarrow \shH_{X,L}$. Notice that the projection $\shH_{X,L} \to X$ is non-flat exactly along the base-locus $X_{L^\perp}$; and that if $X_{L^\perp} = \emptyset$, $\shH_{X,L}$ is nothing but the usual universal hyperplane for $X \to \PP(L^\perp)$, therefore the construction $\shH_{X,L}$ indeed generalizes the construction of universal hyperplane section. 

Let $\shA \subset D(X)$ be a $\PP(V)$-linear Lefschetz category of length $m$, with Lefschetz center $\shA_0$ and components $\shA_k$'s. Then by construction of \S \ref{sec:hyp} we obtain the {\em generalized universal hyperplane} $\shH_{\shA,L} \subset D(\shH_{X,L})$ for $\shA$, with induced adjoint functors:
	$$\begin{tikzcd} \shH_{\shA,L} \ar[shift left]{r}{\delta_{\shH_L \, *}} & \shA \boxtimes D(\PP(L)) \ar[shift left]{l}{\delta_{\shH_L}^*}.
	\end{tikzcd}
	$$

Our main result is the "HPD between linear section and refined blowing up":
\begin{theorem}\label{thm:app:HPDbs} For a $\PP(V)$-linear Lefschetz category $\shA$ of length $m$, with Lefschetz center $\shA_0$ and components $\shA_k$, and a subbundle $L\subset V^\vee$ as above. Consider the ``generalized HPD category" $\shF_L$, i.e. the full $\PP(L)$-linear subcategory of $\shH_{\shA,L}$ defined by: 
	$$\shF_L : = \{C \in \shH_{\shA,L}  \mid \delta_{\shH_L \,*} \,C \in \shA_0 \boxtimes D(\PP(L))\} \subset \shH_{\shA,L}.$$
Then $\shF_L$ is naturally $\PP(L)$-linear equivalent to the linear section $\shA^\hpd_{\PP(L)}$, i.e. the base-change category of the HPD category $\shA^\hpd = (\shA)^\hpd_{/\PP(V)}$ along inclusion $\PP(L) \subset \PP(V^\vee)$. Furthermore, $\shF_L = \shA^\hpd_{\PP(L)}$ is the HPD category (over $\PP(L^\vee)$) of the refined blowing-up category $\rBl_{\shA_{\PP(L^\perp)}} \shA$:
	$$\shA^\hpd_{\PP(L)} \simeq (\rBl_{\shA_{\PP(L^\perp)}} \shA)^\hpd,$$
where $\rBl_{\shA_{\PP(L^\perp)}} \shA$ is equipped with $\PP(L^\vee)$-linear Lefschetz structure given in Prop. \ref{prop:app:lef:ref-bl}. 
\end{theorem}

\begin{remark} 
	$(1)$ The Lefschetz structure on $\shA^\hpd_{\PP(L)}$ obtained from $\shA^\hpd_{\PP(L)} \simeq (\rBl_{\shA_{\PP(L^\perp)}} \shA)^\hpd$ is equivalent to the one given by the fundamental theorem of HPD applied to linear section $\shA^\hpd_{\PP(L)} =\shA^\hpd \boxtimes_{\PP(V^\vee)} D(\PP(L))$. In particular the theorem implies that the linear section $\shA^\hpd_{\PP(L)}$ of HPD contains (exactly) one copy of $\shC_L$, the nontrivial component dual linear section $\shA_{\PP(L^\perp)}$. This is exactly the main statement of the fundamental theorem of HPD. \\
	$(2)$ From HPD is a duality relation (\cite{Kuz07HPD, JLX17}), we also have 
	$\rBl_{\shA_{\PP(L^\perp)}} \shA \simeq (\shA^\hpd_{\PP(L)} )^\hpd_{/ \PP(L)}.$\\
	$(3)$ If $m < \ell$, then $\rBlA = \Bl_{\shA_{\PP(L^\perp)}} \shA$, and the theorem is the generalization of the main result ``the HPD between linear section and blowing up" of \cite{CT15} to categories. If $m \ge \ell$, then $\rBlA \subsetneqq  \Bl_{\shA_{\PP(L^\perp)}} \shA$, and the refinement is necessary for the above HPD statement to hold.
\end{remark}

To prove the theorem, we first fix certain notations for the rest of this section. Set
$$S_0 := \PP(L), \qquad \sE := \Omega^1_{\PP(L)} \otimes \sO_{\PP(L)}(1), \qquad Q_L := \PP_{S_0}(\sE).$$
Notice $Q_L \subset \PP(L) \times \PP(L^\vee)$ (with inclusion induced by the inclusion of vector bundle $\Omega^1_{\PP(L)}  \otimes \sO_{\PP(L)}(1)\subset L^\vee \otimes \sO_{\PP(L)}$ over $S_0$) is nothing but the universal quadric for $\PP(L)$ and $\PP(L^\vee)$. Then the twisted relative cotangent bundle 
	\begin{align}\label{eqn:sM}
	\sM : = \Omega^1_{\PP_{S_0}(\sE)/S_0} \otimes \sO_{\PP_{S_0}(\sE)}(1)
	\end{align}
is the $0$'th cohomology of the complex of vector bundles
	$\sM \simeq \{ \sO_{\PP(L)}(-1) \to L \otimes \sO \to \sO_{\PP(L^\vee)}(1)\}$
by using (dual) Euler sequence twice, and also
	$$\sM^\vee \simeq \{ \sO_{\PP(L^\vee)}(-1) \to L^\vee \otimes \sO \to \sO_{\PP(L^\vee)}(-1)\} \simeq \{\shT_{\PP(L^\vee)}(-1) \to \sO_{\PP(L)}(1)\}.$$ We continue to denote $Q \subset \PP(V) \times \PP(V^\vee)$ for the universal quadric for $\PP(V)$ and $\PP(V^\vee)$.

We further denote for the rest of the section that
$$S := \shH_{\PP(V),L} = \shH_{\PP(L)}/\PP(V^\vee), \qquad Z := \PP(L^\perp) \times \PP(L) \subset S, \qquad \widetilde{S} := \Bl_Z S.$$
Then $Z \subset S$ is the zero locus of a regular section of the (pulling-back of the) vector bundle $\sE \otimes \sO_{\PP(V)}(1)$. The key is to observe that the universal hyperplane $(\shH_{\widetilde{\PP(V)}})_{/\PP(L^\vee)} \subset \widetilde{\PP(V)} \times \PP(L)$ for the blowing up $\widetilde{\PP(V)}\to \PP(L^\vee)$ is nothing but the blowing up $\widetilde{S}$ of $S$. We have a $\PP(V) \times \PP(L)$-linear commutative diagram similar to (\ref{diag:Bl_Z S}) but for the blowing up $\widetilde{S} \to S$, with $E_Z = \PP(L^\perp) \times Q $, where $S \subset  \PP(V) \times \PP(L)$ is a $\sO(1,1)$-divisor. If we base-change these constructions along the natural morphism 
	$$\overline{Y}:= X \times \PP(L) \to  \overline{S}: =\PP(V) \times \PP(L),$$
then we obtain that the universal hyperplane $\widetilde{Y}$ of the blowing up $\tilde{X} = \Bl_{X_{L^\perp} X}$,
	$$\widetilde{Y} := \widetilde{S} \times_{X \times \PP(L) } (\PP(V) \times \PP(L))  \equiv (\shH_{\tilde{X}})_{/\PP(L^\vee)} \subset \widetilde{X} \times \PP(L)$$
is the blowing up of the generalized universal hyperplane 
	$$Y: = S \times_{X \times \PP(L)}  (\PP(V) \times \PP(L)) \equiv  \shH_{X,L} $$
 along $\hat{j} \colon X_{L^\perp} \times \PP(L) \hookrightarrow \shH_{X,L}$. Therefore we have the following $\overline{Y}$-linear commutative diagram for the blowing up $\gamma \colon \widetilde{Y}  \to Y$, with names of morphisms as indicated:
	\begin{equation} \label{digram:bl:Y}
	\begin{tikzcd}[row sep= 2.6 em, column sep = 2.6 em]
	E_{Y_Z} =   X_{L^\perp} \times Q_L \ar{d}[swap]{p_Q} \ar[hook]{r}{j_Q} & \widetilde{Y}  =\shH_{\widetilde{X}}   \ar{d}{\gamma} \ar[hook]{r}{\iota} & \PP_Y(\sE) \ar{ld}[near start]{g} 
	\\
	Y_Z = X_{L^\perp} \times \PP(L) \ar[hook]{r}{\hat{j}}         & Y =  \shH_{X,L}
	\end{tikzcd}	
	\end{equation}	

If we apply the construction of \S \ref{sec:bl} (see also Rmk. \ref{rmk:thm:bl:nreg}) to the $\overline{S}$-linear subcategory $\overline{\shD}: = \shA \boxtimes D(\PP(L)) \subset D(\overline{S})$, we obtain that the universal hyperplane $\shH_{\widetilde{\shA}}$ for the $\PP(L^\vee)$-linear category $\widetilde{\shA}$ is the blowing up category of the generalized universal hyperplane $\shH_{\shA,L}$ along the base-point category
	$(\shH_{\shA,L})_{\PP(L^\perp) \times \PP(L)} = \shA_{L^\perp} \boxtimes D(\PP(L))$. Denote by
	$$\widetilde{\shD} := \shH_{\widetilde{\shA}}, \quad \shD:= \shH_{\shA,L}, \quad \text{and} \quad \shD_{Z} : = (\shD)_{\PP(L^\perp) \times \PP(L)} = \shA_{L^\perp} \boxtimes D(\PP(L)).$$

We fix the following notations for line bundles: denote the pull-backs of the line bundle $\sO_{\PP(L^\vee)}(1)$ to $\widetilde{X}$ and also $\shH_{\widetilde{X}}$ by the same notation $\sL$, by abuse of notations, and the pull-backs 
of line bundles $\sO_{\PP(V)}(1)$ (all factoring through $X \to \PP(V)$) by $\sO_X(1)$. Then 
	$$\sL|_{\widetilde{X}} = \beta^* \sO_X(1) \otimes \sO(-E), \qquad \sL|_{\shH_{\widetilde{X}}} = \gamma^* \sO_X(1) \otimes \sO(-E_Q).$$
To avoid confusions, we denote the induced line bundles on $Q_L \subset \PP(L) \times \PP (L^\vee)$ by:
	$$\sO_{Q_L}(1,0) = (\sO_{\PP(L)}(1) \boxtimes \sO_{\PP(L^\vee)} )|_{Q_L}, \qquad \sO_{Q_L}(0,1) = (\sO_{\PP(L)} \boxtimes \sO_{\PP(L^\vee)}(1) )|_{Q_L}.$$

Notice that the category $\shD= \shH_{\shA,L}$ admits a $\PP(V)$-linear structure from pulling back the $\PP(V)$-linear structure on $\shA \subset D(X)$. By Thm. \ref{thm:hyp}, there is a semiorthogonal decomposition:
	$$ \shD \equiv \shH_{\shA,L} = \langle \shA^\hpd_{\PP(L)},~ \shA_1(1) \boxtimes D(\PP(L)), \ldots, \shA_{m-1}(m-1) \boxtimes D(\PP(L)) \rangle,$$
which coincides of the base-change of Lem. \ref{lem:sodH} along $\PP(L) \subset \PP(V^\vee)$. This implies the fist statement $\shF_L =  \shA^\hpd_{\PP(L)}$ of Thm. \ref{thm:app:HPDbs}. Twisting above $\sO_X(-1)$, then $\shD$ can be regarded to have a {\em $\PP(V)$-linear Lefschetz structure}:
	$$\shD = \langle \shD_0, \shD_1 \otimes \sO_X(1), \ldots, \shD_{m-2} \otimes \sO_X(m-2)\rangle,$$
with Lefschetz components (which are $S_0= \PP(L)$-linear)
	\begin{align}\label{eqn:sod:D_k}
	\shD_k= 
	\begin{cases} \big\langle  \shA^\hpd_{\PP(L)} \otimes \sO_X(-1), ~\shA_1 \boxtimes D(\PP(L)) \big \rangle,  & k=0; \\ 
	\shA_{k+1} \boxtimes D(\PP(L)), &  1\le k \le m-2.
\end{cases}
	\end{align}

We want to show the similar argument of Lem. \ref{lem:app:ref-bl} and Prop. \ref{prop:app:lef:ref-bl} can be applied to the blowing up $\widetilde{\shD}$ of $\shD$, to obtain a semiorthogonal decomposition of $\widetilde{\shD}$ into ambient component and the refined component, and this together with the semiorthogonal decomposition of $\widetilde{\shD} = \shH_{\widetilde{\shA}}$ as the universal hyperplane section yields the desired statement of Thm .\ref{thm:app:HPDbs}.

\begin{lemma} \label{lem:app:lef:proof} The functor $\tilde{\iota}^* \colon \PP_{\widetilde{\shD}}(\sE) \to \widetilde{\shD}$ is fully faithful on subcategories 
	$$\PP_{\gamma^* \shD_{\ell-2}}(\sE), \ldots, \PP_{\gamma^* (\shD_{m-2}(m-\ell))}(\sE),$$
(where $\shD_k$ are $S_0=\PP(L)$-linear categories defined by (\ref{eqn:sod:D_k}), $\gamma^* \colon \shD \to \widetilde{\shD}$ is the blowing up morphism induced by geometric blowing up $\gamma \colon \shH_{\widetilde{X}} \to \shH_{X,L}$, $\sE = \Omega_{\PP(L)}^1(1)$ as before, and the projective bundle category $\PP_{\shB}(\sE)$ for a $S_0$-linear category $\shB$ is defined in \S \ref{sec:app:proj_bd}) and their images form a semiorthogonal sequence in $\widetilde{\shD}$. Denote the span of images by
	$$\widetilde{\shD}^{\rm amb} : = \big \langle \tilde{\iota}^*(\PP_{\gamma^* \shD_{\ell-2}}(\sE)), \ldots, \tilde{\iota}^*(\PP_{\gamma^* (\shD_{m-2}(m-\ell))}(\sE)) \subset \widetilde{\shD},$$ 
and its right orthogonal by $\widetilde{\shD}^{\rm ref} : = (\widetilde{\shD}^{\rm amb})^\perp \subset \widetilde{\shD}$. 
Then $\widetilde{\shD}^{\rm ref}$ admits a $\PP(L^\vee)$-linear Lefschetz structure 
	$$\widetilde{\shD}^{\rm ref}  = \langle \widetilde{\shD}^{\rm ref}_0, \widetilde{\shD}^{\rm ref}_1 \otimes \sL, \ldots, \widetilde{\shD}^{\rm ref}_{\ell-3} \otimes \sL^{\ell-3}\rangle,$$
where $\widetilde{\shD}^{\rm ref}_k = \widetilde{\shD} \cap \shD_{\ell-2}^\perp$ is equivalent to 
	$$ \widetilde{\shD}^{\rm ref}_k= 
	\begin{cases} \Big \langle  \gamma^*\langle \shA^\hpd_{\PP(L)} \otimes \sO_X(-1),\shA'_1 \boxtimes D(\PP(L)) \rangle,~~ (\shC_L \boxtimes D(\PP(L)))_0 \big \rangle, & k=0; \\ 
	\Big \langle \gamma^*\big(\shA'_{k+1} \boxtimes D(\PP(L))\big), ~~ (\shC_L \boxtimes D(\PP(L)))_0 \Big \rangle,  &  1\ge k \ge \ell-3,
\end{cases}$$
where $\shA_k' : = \shA_k \cap \shA_{\ell-1}^\perp$ is defined by the same formula as Prop. \ref{prop:app:lef:ref-bl}.
\end{lemma}

This lemma will be proved later. We show that Thm. \ref{thm:app:HPDbs} can be deduced from this lemma:
\begin{proof}[Proof of Thm. \ref{thm:app:HPDbs}]
	Notice the universal hyperplane (\S \ref{sec:app:hyp}) for the $\PP(L^\vee)$-linear decomposition $\widetilde{\shA} = \langle \rBlA, \widetilde{\shA}^{\rm amb}\rangle$ of Prop. \ref{prop:app:lef:ref-bl} admits a $\PP(L^\vee)$-linear decomposition
	$$\widetilde{\shD} = \shH_{\widetilde{\shA} } = \langle \shH_{\rBlA}, \shH_{\widetilde{\shA}^{\rm amb}} \rangle.$$
Compare $\shH_{\widetilde{\shA}^{\rm amb}}$ with $\widetilde{\shD}^{\rm amb}$ of Lem. \ref{lem:app:lef:proof}, notice that $\shH_{\widetilde{\shA}^{\rm amb}} = \PP_{\widetilde{\shA}^{\rm amb}} (\sE)$, where $\sE = \Omega^1_{\PP(L)}(1)$ as before, one sees directly that
	$\widetilde{\shD}^{\rm amb} = \shH_{\widetilde{\shA}^{\rm amb}}$. 
Hence we have a $\PP(L)$-linear equivalence $\shH_{\rBlA}  = \widetilde{\shD}^{\rm ref} = \widetilde{\shD}^{\rm ref} \otimes \sL$.  Now from the defining property (Lem. \ref{lem:sodH}) of HPD, there is a $\PP(L)$-linear semiorthogonal decomposition:
	$$\shH_{\rBlA} = \langle (\rBlA)^\hpd, ~ \rBlA_1(\sL) \boxtimes D(\PP(L)), \ldots, \rBlA_{\ell-2}(\sL^{\ell-2}) \boxtimes D(\PP(L))\rangle.$$
If we compare the above semiorthogonal decomposition with the one for $\shH_{\rBlA} = \widetilde{\shD}^{\rm ref} \otimes \sL$ from Lem. \ref{lem:app:lef:proof}, we obtain directly that the $\PP(L)$-linear functor
	$\Phi \colon \shH_{\shA,L} \to \shH_{\widetilde{\shA}}$,
	$$\Phi (A)= \gamma^*(A  \otimes \sO_X (-1)) \otimes \sL = \gamma^*A \otimes \sO(-E_{Y_Z})  \quad \text{for} \quad A \in  \shH_{\shA,L} ,$$
induces an equivalence of categories $\shA_{\PP(L)}^\hpd \simeq (\rBlA)^\hpd$. \end{proof}

\begin{remark} The twisting $\otimes \sO_X(-1)$ and $\otimes \sL$ in the expression of $\Phi$ is only a matter of convention. In fact, for a $\PP(V)$-linear Lefschetz category $\shA$, there is a definition of a {\em left} HPD category ${}^{\hpd} \shA$, which related to the usual (right) HPD category $\shA^\hpd$ by the $\PP(V^\vee)$-linear equivalence ${}^{\hpd} \shA = \shA^\hpd \otimes \sO_{\PP(V)}(-1)$ (see \cite{P18}). Then the theorem can reformulated as: the pullback functor $\gamma^* \colon \shH_{\shA,L} \to \shH_{\widetilde{\shA}}$ of the blowing up $\gamma$ induces an equivalence for left HPDs:
	$$\gamma^* \colon ({}^\hpd \shA)_{\PP(L)} \simeq {}^\hpd (\rBlA).$$ 
\end{remark}

\subsubsection*{Observations} It remains to prove Lemma \ref{lem:app:lef:proof}. We first make the following observations.
\begin{enumerate}[leftmargin = *]
	\item  \label{obs:1} 
From $\sE = \Omega_{\PP(L)}^1(1)$, $\PP_{\PP(L)}(\sE) = Q_L \subset \PP(L) \times \PP(L^\vee)$, and that $\widetilde{Y} \subset \widetilde{X} \times \PP(L)$ is the universal hyperplane, then by construction the projective bundle $\PP_{\widetilde{Y}}(\sE)$ fits into following diagram of embeddings:
\begin{equation*}
	\begin{tikzcd} 
	\PP_{\widetilde{Y}}(\sE) \ar[hook]{d}[swap]{\theta_1} \ar[hook]{r}{\theta_2}  & \widetilde{X}  \times Q_L  \ar[hook]{d}{ \Id  \times \delta_{Q_L}} \\
	\shH_{\widetilde{X}} \times \PP(L^\vee) \ar[hook]{r}{\delta_{\shH} \times \Id}         &  \widetilde{X} \times \PP(L)  \times \PP(L^\vee). 
	\end{tikzcd}	
	\end{equation*}
where the inclusion $\theta_1$ (resp. $\theta_2$) is an inclusion of divisor of $\sO_{\PP(L)}(1) \boxtimes \sO_{\PP(L^\vee)}(1)$ (resp. $\sO_{\widetilde{X}}(1) \boxtimes \sO_{\PP(L)}(1)$), and $\delta_{Q_L} \colon Q_L \to \PP(L) \times \PP(L^\vee)$ and $\delta_{\shH} \colon \widetilde{X} \times \PP(L)$ denote the inclusion of universal hyperplanes as usual. Then the funtors $\theta_1^*$ (resp. $\theta_2^*$) induces functors $\theta_1^* \colon \shH_{\widetilde{\shA}} \boxtimes D(\PP(L^\vee)) \to \PP_{\widetilde{\shA}}(\sE)$ (resp. $\theta_2^* \colon \widetilde{\shA} \boxtimes D(Q_L) \to \PP_{\widetilde{\shA}}(\sE)$).

	\item \label{obs:2} Next notice that the graph embedding $\widetilde{Y} = \shH_{\widetilde{X}} \hookrightarrow  \widetilde{Y} \times \PP(L^\vee)$ factors through $\widetilde{\iota} \colon \widetilde{Y} \hookrightarrow \PP_{\widetilde{Y}}(\sE)$, and the inclusion $\PP_{\widetilde{Y}}(\sE)\subset  \widetilde{Y} \times \PP(L^\vee)$. $\widetilde{\iota} $ is a lift of the embedding $\iota \colon \widetilde{Y}   \hookrightarrow \PP_{Y}(\sE)$ along the natural projection $\PP_{\widetilde{Y}}(\sE) \to \PP_{Y}(\sE)$. Therefore we have a commutative diagram:
	\begin{equation} \label{eqn:app:bl-graph:Y}
	\begin{tikzcd} &\PP_{\widetilde{Y} }(\sE) \ar{d}  \ar[hook]{r}{\theta_2}& \widetilde{X} \times Q_L \ar{d}{\beta \times \Id} \\
	\widetilde{Y}  =  \shH_{\widetilde{X}} \ar[hook]{ur}{\widetilde{\iota}} \ar[hook]{r}{\iota}& \PP_{Y}(\sE)  \ar[hook]{r}{\overline{\theta}_2} & X \times Q_L 
	\end{tikzcd}
	\end{equation}
which will play the role of diagram (\ref{eqn:app:bl-graph}) in last subsection, where the inclusion $\overline{\theta}_2 \colon \PP_Y(\sE) \hookrightarrow \PP_{X \times \PP(L)}(\sE) = X \times Q_L $ is induced from the inclusion $Y = \shH_{X,L} \hookrightarrow X \times \PP(L)$ by generalized universal hyperplane construction. Since the graph embedding $\widetilde{Y} = \shH_{\widetilde{X}} \hookrightarrow  \widetilde{Y} \times \PP(L^
\vee)$ is given by a regular section of the vector bundle $\sL \boxtimes \shT_{\PP(L^\vee)}(-1)$, and $\PP_{\widetilde{Y}}(\sE) \subset \widetilde{Y} \times \PP(L^\vee)$ is a divisor of $\sL \otimes \sO_{\PP(L)}(1)$, therefore the normal bundle of $\tilde{\iota}$ is 
	$$N_{\tilde{\iota}} = \sL \otimes \sM^\vee,$$
where $\sM$ is the rank $\ell-1$ vector bundle defined by (\ref{eqn:sM}). Notice that $\tilde{\iota}^* \colon D(\PP_{\widetilde{Y} }(\sE)) \to D(\widetilde{Y})$ induces a functor $\tilde{\iota}^* \colon \PP_{\widetilde{\shD}}(\sE) \to \widetilde{\shD} = \shH_{\widetilde{\shA}}$. Notice also that the line bundles have the following identifications under the above morphisms:
	$$\sL|_{\widetilde{Y}} = \tilde{\iota}^* \, \theta_2^*( \sO_{Q_L}(0,1)) = \iota^* \,\overline{\theta}_2^* (\sO_{Q_L}(0,1)).$$

	\item \label{obs:3} $\widetilde{\shD}$ as the blowing up category of $\shD$ along $\shD_Z$ admits a $\overline{Y}= X \times \PP(L)$-linear semiorthogonal decomposition from blowing up formula Thm. \ref{thm:bl}:	
	\begin{align*} \widetilde{\shD}  & = \langle \gamma^* \shD, ~ (\shD_{Z})_0, (\shD_{Z})_1, \ldots, (\shD_{Z})_{\ell-3} \rangle 
				= \langle (\shD_{Z})_{2-\ell}, \ldots, (\shD_{Z})_{-2}, (\shD_{Z})_{-1}, ~ \gamma^* \shD \rangle, 			
	\end{align*}
(notice that the codimension of the center of blowing up is now $\ell-1$ instead of $\ell$), where $(\shD_{Z})_k$ denotes the image of $\shD_{Z}$ under the fully faithful embedding $j_{Q*} p_Q^* (-) \otimes \sL^k$, $k \in \ZZ$.

	\item \label{obs:4} Assume $\shB \subset \shA$ is an admissible subcategory such that $\Hom(\shB, \shB\otimes \sO_X(-1)) = 0$ (this holds for example for any Lefschetz components $\shA_k \subset \shA$ with $k \ge 1$ and their twists by $\sO_X(t)$, $t \in \ZZ$), then the pullback $\delta_{\shH_{L}}^* \colon \shA \boxtimes D(\PP(L)) \to \shH_{\shA,L}$ of the $\sO(1,1)$-divisor inclusion $\delta_{\shH_{L}}$ is fully faithful on the subcategory $\shB \boxtimes D(\PP(L))$; and similarly the functor $\theta_2^* \colon \widetilde{\shA} \boxtimes D(Q_L) \to \PP_{\widetilde{\shA}}(\sE)$ is fully faithful on $(\beta^* \shB) \boxtimes D(Q_L)$. Furthermore, we have natural equivalence of categories
	$$\PP_{\gamma^*\, \delta_{\shH_L}^*(\shB \boxtimes D(\PP(L)))}(\sE) = \theta_2^* (\beta^* \shB \boxtimes D(Q_L)) \subset  \PP_{\widetilde{\shA}}(\sE).$$
If we apply above to (\ref{eqn:sod:D_k}), we have  in particular the following identifications
	$$\PP_{\gamma^* \shD_{k}(m_k)}(\sE) = \theta_2^* (\beta^* \shA_{k+1}(m_k) \boxtimes D(Q_L)), \qquad 1 \le k \le m-2, ~\forall m_k \in \ZZ,$$
and $\PP_{\gamma^* (\shA_1(m_1) \boxtimes D(\PP(L)) )} = \theta_2^* (\beta^* \shA_{1}(m_1) \boxtimes D(Q_L))$ for all $m_1 \in \ZZ$.

\end{enumerate}

\medskip
The proof Lem. \ref{lem:app:lef:proof}, similar to that of Prop. \ref{prop:app:lef:ref-bl}, can be decomposed into several steps.


\begin{lemma}[cf. Lem. \ref{lem:app:ref-bl}] \label{lem1} The functor $\tilde{\iota}^*$ is fully faithful on the following subcategories
	$$\PP_{\gamma^* (\shD_{k-1}(k+1-\ell))}(\sE) =\theta_2^*(\beta^* \shA_{k}(k+1-\ell) \boxtimes D(Q_L)), \qquad \ell -1 \le k \le m-1,$$
and their images again form a semiorthogonal sequence in $\widetilde{\shD}$.
\end{lemma}
\begin{proof}
From observation (\ref{obs:2}), $N_{\tilde{\iota}} = \sL \otimes \sM^\vee$, the functor $\tilde{\iota}_*\, \tilde{\iota}^*$ is the iterated cone of functors:
	$$\otimes \sL^{-r} \otimes \wedge^r \sM, \quad 0 \le r \le \ell-2,$$
where $ \sL^{-r}:=(\sL^\vee)^{\otimes r}$ as before, $\sM$ is the rank $\ell-1$ vector bundle defined by (\ref{eqn:sM}), and the case $r=0$ correspond to identity functor. Therefore for any $A_1,A_2 \in \widetilde{\shA}$, $F_1,F_2 \in D(Q_L)$, 
	$$\Hom\big(\tilde{\iota}^*(\theta_2^*(A_1 \boxtimes F_1)), \tilde{\iota}^*((\theta_2^*(A_2 \boxtimes F_2))\big)$$
is an iterated cone of the $\Hom$ spaces:
	\begin{align}
	  &\Hom\big (\theta_2^*((A_1 \otimes \sL^r) \boxtimes F_1), \theta_2^*(A_2 \boxtimes (F_2 \otimes \wedge^r \sM ) \big), &0\le r \le \ell-2,
	\end{align}
where the case $r=0$ corresponds to $\Hom(\theta_2^*(A_1 \boxtimes F_1), \theta_2^*(A_2 \boxtimes F_2))$. Then if $A_{k} \in \beta^*\shA_{i_k}(i_k-\ell)$, $k=1,2$, such that $\ell -1 \le i_2 \le i_1 \le m-1$, then
	$$\theta_2^*(A_2 \boxtimes (F_2 \otimes \wedge^r \sM)) \in \PP_{\beta^*\shA_{i_2}(i_2+1-\ell) \boxtimes D(\PP(L))} (\sE) = \PP_{\gamma^* \shD_{i_2-1}(i_2 +1 -\ell)}(\sE),$$ 
and the desired result follows from the semiorthogonality of subcategories of $\PP_{\widetilde{\shD}}(\sE)$:
	$$\big(\PP_{\gamma^* \shD_{i_2-1}(i_2 +1 -\ell)}(\sE) , ~~ \PP_{\gamma^* (\shD_{i_1-1}(i_1 +1 -\ell) )\otimes \sL^r }(\sE)
	 \big),$$
for $1 \le r \le \ell-1$, which follows from the the semiorthogonality of the subcategories
	$(\beta^* (\shA_{i_2}(i_2+1- \ell)), \beta^*(\shA_{i_1}(i_1+1-\ell)) \otimes \sL^r  )$ 
of $\widetilde{\shA}$ of Lem. \ref{lem:app:ref-bl}.
\end{proof}

\begin{lemma}[cf. Prop. \ref{prop:app:lef:ref-bl}] \label{lem2}
Define $\widetilde{\shD}^{\rm amb}$ to be the category generated by the images of above Lemma \ref{lem1}, and let $\widetilde{\shD}^{\rm ref}$ be defined as in the statement of Lem. \ref{lem:app:lef:proof}. Then
	$$\widetilde{\shD}_0^{\rm ref}, \widetilde{\shD}_1^{\rm ref} \otimes \sL, \cdots, \widetilde{\shD}_{\ell-3}^{\rm ref} \otimes \sL^{\ell-3} $$
forms a semiorthogonal sequence in the subcategory $(\widetilde{\shD}^{\rm amb})^\perp$. 
\end{lemma}
\begin{proof}
The proof is similar to Step 1, Step 2 of the proof of Prop. \ref{prop:app:lef:ref-bl}. The fact that they are semiorthogonal follows directly from the semiorthogonal decomposition of $\widetilde{\shD}$ (see observation (\ref{obs:3}) above). The key part of this step is to show they all belong to $(\widetilde{\shD}^{\rm amb})^\perp$. 

First, we show the following components
	$$\langle \gamma^*(\shA_1' \boxtimes D(\PP(L))), \ldots, \gamma^*(\shA_{\ell-2}' \boxtimes D(\PP(L))) \otimes \sL^{\ell-3} \rangle \subset (\widetilde{\shD}^{\rm amb})^\perp.$$
This is equivalent to show for any 
$A_1 \in \beta^*(\shA_{i_1}(i_1 -\ell +1))$, $F_1 \in D(Q_L)$ as previous step, $A_2 \in \beta^*(\shA_{i_2}') = \beta^*(\shA_{\ell-1}^\perp \cap \shA_{i_2})$, $F_2 \in D(\PP(L))$, where $1 \le i_2 < \ell-1 \le i_1 \le m-1$,
	$$\Hom(\tilde{\iota}^* \theta_2^* (A_1 \boxtimes F_1), \gamma^*(A_2 \boxtimes F_2) \otimes \sL^{i_2-1}) = 0.$$
Denote the composition of natural projections by $\tilde{g} \colon \PP_{\widetilde{Y}}(\sE) \to \widetilde{Y} \to Y$, then $\tilde{g}$ also factorizes through natural projections $ \PP_{\widetilde{Y}}(\sE) \to \PP_{Y}(\sE) \to Y$, and $\gamma = \tilde{\iota} \circ \tilde{g} \colon Y \to \widetilde{Y}$. Therefore
	$$\gamma^*(A_2 \boxtimes F_2) \otimes \sL^{i_2-1} = \tilde{\iota}^* (\tilde{g}^*A_2 \otimes F_2 \otimes \sO_{Q_L}(0,i_2-1)) =  \tilde{\iota}^*\,\theta_2^*(A_2 \boxtimes F_2(0,i_2-1) ),$$
where $F_2(0,i_2-1)$ denotes $F_2 \otimes \sO_{Q_L}(0,i_2-1))|_{Q_L} \in D(Q_L)$. Therefore the above $\Hom$ space is an iterated cone of 
	$$\Hom\big( \theta_2^* ((A_1 \otimes \sL^r) \boxtimes F_1), \theta_2^*(A_2 \boxtimes (F_2 \otimes \wedge^r \sM(0,i_2-1))) \big), \qquad 0 \le r \le \ell-2.$$
The latter components of the above $\Hom$ space for a given $r$ belongs to
	$$\theta_2^*(\beta^* \shA_{i_2}' \boxtimes D(Q_L)) = \PP_{\gamma^*(\shA_{i_2}' \boxtimes D(\PP(L))} (\sE),$$
and the former components of the $\Hom$ space belongs to
	$$\theta_2^*((\beta^* \shA_{i_1} \otimes \sL^r) \boxtimes D(Q_L)) = \PP_{\gamma^*((\shA_{i_1} \otimes \sL^r) \boxtimes D(\PP(L))} (\sE).$$
Therefore the desired semiorthogonality follows from $\big(\PP_{\gamma^*(\shA_{i_2}' \boxtimes D(\PP(L))} (\sE), \PP_{\gamma^*((\shA_{i_1} \otimes \sL^r) \boxtimes D(\PP(L))} (\sE)\big)$ is a semiorthogonal sequence of $\PP_{\widetilde{\shD}}(\sE)$, which follows from the 
fact that
	$$\big(\beta^*\shA_{i_2}', \beta^*(\shA_{i_1}(i_1+1-\ell)) \otimes \sL^r \big)$$	
is semiorthogonal for all $1 \le i_2 < \ell-1 \le i_1 \le m-1$ and $0 \le r \le \ell-2$ proved in Step 1 of proof of Prop. \ref{prop:app:lef:ref-bl}. 

Second,
	$$\gamma^*(\shA^\hpd_{\PP(L)}\otimes \sO_X(-1)) \subset (\widetilde{\shD}^{\rm amb})^\perp.$$
This follows from the same argument and that $\big(\PP_{\gamma^*(\shA^\hpd_{\PP(L)}\otimes \sO_X(-1))}(\sE), \PP_{\shD_{i_1} \otimes \sL^{i_1+1-\ell}} (\sE) \big)$ is semiorthogonal for all $\ell-1 \le i_1 \le m-1$.

Third, it remains to show that 
	$$\big\langle (\shC_L \boxtimes D(\PP(L)))_0, \ldots, (\shC_L \boxtimes D(\PP(L)))_{0} \otimes \sL^{\ell-3}  \big\rangle \subset (\widetilde{\shD}^{\rm amb})^\perp,$$
where $(\shC_L \boxtimes D(\PP(L))_0$ denotes the image of $\shC_L \boxtimes D(\PP(L)) \subset \shD_{Z}$ under the fully faithful embedding $j_{Q*} p_Q^* \colon \shD_Z \to \widetilde{\shD}$, following the usual convention for blowing up formula. As the Step 2 of the proof of Prop. \ref{prop:app:lef:ref-bl}, this is equivalent to show
	\begin{equation} \label{eqn:Hom:proof_bs}
	\Hom(\tilde{\iota}^*\, \theta_2^*(\beta^*A \boxtimes F_1), ~ j_{Q*} p_Q^* (C \boxtimes F_2) \otimes \sL^r) =0, \qquad 0 \le r \le \ell-3,
	\end{equation}
for all $A \in \shA_{i}(i+1-\ell)$ where $\ell -1 \le i \le m-1$, $F_1 \in D(Q_L)$, $F_2 \in D(\PP(L))$, $C \in \shC_L$.  

We show this separately for the case $i = \ell-1$ and the case $\ell \le i \le m-1$. For $i = \ell-1$, this follows from the semiorthogonality of the subcategories
	$$\big( \PP_{\gamma^* \shD_{\ell-2}} (\sE), ~ \PP_{(\shD_Z)_0 \otimes \sL^s} (\sE)\big), \qquad 0 \le s \le \ell-3,$$
which follows from the semiorthogonality of subcategories $(\gamma^* \shD_{\ell-2},(\shD_Z) \otimes \sL^s)$ of $\widetilde{\shD}$. The latter fact, analogous to (\ref{eqn:app:van:rbl-bs}) in the case $k=r=\ell-1$, follows directly from applying mutations to the semiorthogonal decomposition of $\widetilde{\shD}$. 

Now we focus on the case $i = \ell, \ldots, m-1$. From the commutative diagram (\ref{eqn:app:bl-graph:Y}), the former factor of the $\Hom$ space in (\ref{eqn:Hom:proof_bs}) is:
	$$\tilde{\iota}^*\, \theta_2^*(\beta^*A \boxtimes F_1) = \iota^* \, \overline{\theta}_2^* (A \boxtimes F_1).$$
Therefore from adjunction, the desired vanishing of (\ref{eqn:Hom:proof_bs}) is equivalent to
	$$\Hom_{\shA \boxtimes D(Q_L)}\big(A \boxtimes F_1, ~ \overline{\theta}_{2*} \iota_*\,( j_{Q*} p_Q^* (C \boxtimes F_2) \otimes \sL^r) \big)=0, \qquad 0 \le r \le \ell-3,$$
From the ambient square of (\ref{digram:bl:Y}) is Tor-independent, we have $ \iota_* j_{Q*} p_Q^* = g^* \, \hat{j}_* $, therefore the latter factor of above $\Hom$ space is
	\begin{align*}
	& \overline{\theta}_{2*} \,\iota_*( j_{Q*} p_Q^* (C \boxtimes F_2) \otimes \sL^r) =  \overline{\theta}_{2*} \,( \iota_* j_{Q*} p_Q^* (C \boxtimes F_2) \otimes \overline{\theta}_2^*(\sO_{Q_L}(0,r)))\\
	&= \overline{\theta}_{2*} g^* \, \hat{j}_*  (C \boxtimes F_2) \otimes \sO_{Q_L}(0,r) = (\Id_X \times \pi_{Q_L})^* \delta_{\shH*} \,\hat{j}_*  (C \boxtimes F_2)) \otimes \sO_{Q_L}(0,r),
		\end{align*}
where $\delta_{\shH} \colon Y \times X \times \PP(L)$ is the inclusion, $\pi_{Q_L}$ denotes the projection $Q_L \to \PP(L)$, $\overline{\theta}_{2*} g^* =  (\Id_X \times \pi_{Q_L})^* \,\delta_{\shH*} $ since the corresponding square for projective bundle of $\sE$ is flat. Since
	$$\delta_{\shH} \circ \hat{j} = i_{L} \times \id_{\PP(L)} \colon X_{L^\perp} \times \PP(L)  \hookrightarrow X \times \PP(L),$$
where $i_{L} \colon X_{L^\perp} \hookrightarrow X$ is the inclusion as before, therefore the above factor is isomorphic to
	$$i_{L*} (C) \boxtimes (\pi_{Q_L}^* F_2 \otimes \sO_{Q_L}(0,r)) \in i_{L*} \shC_L \boxtimes D(Q_L) \subset \shA \boxtimes D(Q_L).$$
Now the desired vanishing is equivalent to
	$$\Hom_{\shA} (A,i_{L*} (C)) \otimes \Hom_{Q_L}(F_1, \pi_{Q_L}^* F_2 \otimes \sO_{Q_L}(0,r))= 0$$
which follows from the definition of $\shC_L$, i.e. $\Hom_{\shA}(A, i_{L*} C) = \Hom_{\shA_{L^\perp}}(i_L^* A, C)= 0$ for $A \in \shA_{i}(i+1-\ell)$, $i=\ell, \ldots, m-1$, $C \in \shC_L$.
\end{proof}

\begin{lemma}[Generation; cf. Prop. \ref{prop:app:lef:ref-bl}] \label{lem3} The following semiorthogonal sequence
	$$\widetilde{\shD}_0^{\rm ref}, \widetilde{\shD}_1^{\rm ref} \otimes \sL, \cdots, \widetilde{\shD}_{\ell-3}^{\rm ref} \otimes \sL^{\ell-3} $$
of subcategories of $(\widetilde{\shD}^{\rm amb})^\perp$ from Lemma \ref{lem2} generate $(\widetilde{\shD}^{\rm amb})^\perp$.
\end{lemma} 
\begin{proof}
By comparing the two semiorthogonal decompositions of $\widetilde{\shD}$, one from blowing up formula for $\widetilde{\shD}$ (see observation (\ref{obs:3})) and one is $\widetilde{\shD} = \langle (\widetilde{\shD}^{\rm amb})^\perp, \widetilde{\shD}^{\rm amb} \rangle$, we only need to show the subcategory
	$$j_{Q*} p_Q^* (\hat{j}^*(\delta_{\shH}^*(\shA_{i}(i+1-\ell) \boxtimes D(\PP(L))) \otimes \sL^r, \qquad \ell \le i \le m-1, ~0 \le r \le \ell-3$$
belongs to the category generated by $\widetilde{\shD}^{\rm amb}$ and $\widetilde{\shD}_{i}^{\rm ref} \otimes \sL^{i}$'s, where $i=0,\ldots, \ell-3$. Since the above category is generated by elements of the form
	\begin{align*}
	& j_{Q*} p_Q^* \hat{j}^* \delta_{\shH}^* (A \boxtimes F) \otimes \sL^r= j_{Q*} j_Q^* \, \iota^* \overline{\theta}_2^* \big(A \boxtimes (\pi_{Q_L}^* F \otimes \sO_{Q_L}(0,r)) \big) \\ 
	&=  j_{Q*} j_Q^* \,  \tilde{\iota}^*  \theta_2^*(\beta^* A \boxtimes (F(0,r)))
	\end{align*}
for all $A \in \shA_i(i+1-\ell)$, $F \in D(\PP(L))$, $\ell \le i \le m-1$, $0 \le r \le \ell-3$, we only need to show above elements belongs to the category generated by $\widetilde{\shD}^{\rm amb}$ and $\widetilde{\shD}_{i}^{\rm ref} \otimes \sL^{i}$'s, $i=0,\ldots, \ell-3$, where $F(0,r)$ denotes $pi_{Q_L}^* F \otimes \sO_{Q_L}(0,r) \in D(Q_L)$. Then from the distinguished triangle of functors:
	$$\otimes \sO(-E_Q) \to \Id \to j_{Q*} j_Q^* \xrightarrow{[1]}{},$$
the element $j_{Q*} p_Q^* \hat{j}^* \delta_{\shH}^* (A \boxtimes F) \otimes \sL^r$ is isomorphic to the cone of
	$$\tilde{\iota}^* \theta_2^*(\beta^* A \boxtimes (F(0,r))) \otimes \sO(-E_Q) = \tilde{\iota}^* \theta_2^*(\beta^*(A(-1)) \boxtimes F(0,r+1)),$$
where $A(-1) = A \otimes \sO_{X}(-1) \in  \shA$, and
	$$\tilde{\iota}^* \theta_2^*(A \boxtimes (F(0,r))).$$
Since $A \in  \shA_i(i+1-\ell)$, $A(-1) \subset \shA_{i}(i-\ell) \subset \shA_{i-1}(i-\ell)$, for $\ell \le i \le m-1$, then the terms of the above cone belongs to $\tilde{\iota}^* \PP_{\gamma^* \shD_{i-2}(i-\ell)}$ and respectively $\tilde{\iota}^* \PP_{\gamma^* \shD_{i-1}(i+1-\ell)}$. In particular they all belong to $\widetilde{\shD}^{\rm amb}$. Hence we are done.
\end{proof}

Now Lem. \ref{lem:app:lef:proof} follows directly from Lem. \ref{lem1}, \ref{lem2} and \ref{lem3}.

\begin{remark} Notice for the blowing up $\widetilde{Y} \to Y$, the normal bundle of the embedding $\iota \colon \widetilde{Y} \to \PP_{Y}(\sE)$ is $N_{\iota} = \shT_g(-1)|_{\widetilde{Y}} = \sO_X(1) \otimes \sM^\vee$, where $g \colon \PP_Y(\sE) \to Y$ the projection. Hence in the above proof of Lem. \ref{lem:app:lef:proof}, one can equivalent argue using the Koszul complex for $N_{\iota}$ instead of $N_{\tilde{\iota}}$ (of observation (\ref{obs:2})).
\end{remark}

\subsection{Applications to Calabi--Yau category fibrations} \label{sec:CY}
For integers $d, k \ge 2$, denote $V_{kd}$ a $\kk$-vector space of dimension $k\cdot d$. Consider the Veronese map $v_d \colon X: = \PP(V_{kd}) \hookrightarrow \PP(\Sym^{d} V_{kd})$. Denote $H$ the hyperplane class of $\PP(\Sym^{d} V_{kd})$, then $X$ has a rectangular Lefschetz decomposition:
	$$D(X) = \langle \shA, \shA(H), \ldots, \shA(kH)\rangle, \quad \shA = \langle \sO_X, \sO_X(1), \ldots, \sO_X(d-1)\rangle,$$
where $\sO(H)|_X = v_d^* (\sO_{\PP(\Sym^{d} V_{kd})}(1)) = \sO_{X}(d)$. Denote $\shA^{\hpd}/\PP(\Sym^{d}V_{kd}^*)$ the HPD category of $v_d \colon X \to  \PP(\Sym^{d} V_{kd})$. Then the fiber of $\shA^{\hpd}$ over a general point $[H] \in \PP(\Sym^{d}V_{kd}^*)$ is the essential part of the derived category of a degree $d $ hypersurface $X_H \subset \PP(V_{kd})$:
	$$D(X_H) = \langle \shA^{\hpd}|_{[H]}, ~ \shA, \shA(H), \ldots, \shA((k-1)H)  \rangle,$$
where $\shA^{\hpd}|_{[H]}$ is a Calabi--Yau category of dimension $k(d-2)$ by \cite{KuzCY}. Take a general $k$-dimensional linear subspace $L_k \subset \Sym^{d}V_{kd}^*$, then the base locus is a complete intersection of $k$ degree $d$ hypersurfaces inside $\PP(V_{kd})$:
	$$X_{L_k^\perp} = H_1\cap \cdots \cap H_k =: X_{(d,d,\ldots, d)} \subset \PP(V_{kd}), \quad H_1,\ldots, H_k \in |\sO_X(d)|.$$
i.e. an intersection of type $(d,d,\ldots,d)$, which is a Calabi--Yau manifolds of dimension $kd-k-1$. By Prop. \ref{prop:app:lef:ref-bl}, the refined blowing up has a rectangular Lefschetz decomposition:
	$$D(\Bl^{\rm ref}_{X_{L_k^\perp}} \PP(V_{kd}) )= \big\langle D(X_{L_k^\perp}), D(X_{L_k^\perp})\otimes \sO_{\PP(L_k^*)}(1), \ldots, D(X_{L_k^\perp})\otimes \sO_{\PP(L_k^*)}(k-2) \big\rangle.$$

On the other hand, the linear restriction $\shA^{\hpd}|_{\PP(L_k)}$ is a family of Calabi--Yau cateogires of dimension $k(d-2)$ over the projective space $\PP(L_k) = \PP^{k-1}$. If we denote by:
	$$\delta_{\shH_{L_k}} \colon \shH_{L_k}: = \shH_{X, \PP(L_k)} \hookrightarrow \PP(V_{kd}) \times \PP(L_k)$$
the inclusion of universal family of degree $d$ hypersurface over the linear system $\PP(L_k)$, which is a degree $(d,1)$ hypersurface of $\PP(V_{kd}) \times \PP(L_k)$. Then by first part of main Thm. \ref{thm:app:HPDbs}, $\shA^{\hpd}|_{\PP(L_k)}$ also admits the following descriptions:
	\begin{align*}
	\shA^{\hpd}|_{\PP(L_k)} & =  \{C \in D(\shH_{L_k}) \mid \delta_{\shH_{L_k} *}(C) \in \shA \boxtimes D(\PP(L_k)) \} \\
	& =\left\langle \shA(H) \boxtimes D(\PP L_k) |_{\shH_{L_k}}, \ldots,  \shA((k-1)H) \boxtimes D(\PP L_k )|_{\shH_{L_k}}\right\rangle^\perp \subset D(\shH_{L_k}).
	\end{align*} 
Finally the HPD statement of Thm. \ref{thm:app:HPDbs} implies that:
	$$D(X_{L_k^\perp}) \simeq \shA^{\hpd}|_{\PP(L_k)}.$$

If we take $k=2$, then above implies that the Calabi--Yau $(2d-3)$-fold $X_{(d,d)} \subset \PP^{2d-1}$ of intersection type $(d,d)$ is derived equivalent to a pencil $\shA^{\hpd}|_{\PP^1}$ of Calabi--Yau categories of dimension $2(d-2)$. For example, 
\begin{itemize}
	\item If we take $d=3$, then the Calabi--Yau threefolds $X_{(3,3)} \subset \PP^4$ is derived equivalent to a pencil of K3 categories from cubic fourfolds. This is the example considered by Calabrese--Thomas \cite{CT16}. 
	\item Similarly, the Calabi--Yau $5$-fold $X_{(4,4)} \subset \PP^{7}$ is derived equivalent to a pencil of Calabi--Yau categories of dimension $4$, etc.
\end{itemize}

If $k=3$, then the Calabi--Yau $(3d-4)$-fold $X_{(d,d,d)} \subset \PP^{3d-1}$ is derived equivalent to a fibration $\shA^{\hpd}|_{\PP^2}$ of Calabi--Yau categories of dimension $3(d-2)$ over $\PP^2$. For example, 
\begin{itemize}
	\item The K3 surface $X_{(2,2,2)}\subset \PP^{5}$ is derived equivalence to a CY $0$-category fibration over $\PP^2$, which is nothing but the double cover of $\PP^2$ ramified over a sextic curve.
	\item The Calabi--Yau $5$-fold $X_{(3,3,3)}\subset \PP^{8}$ admits a CY $3$-category fibration over $\PP^2$. 
\end{itemize}
On the other hand, 
\begin{itemize}
	\item If we take $d=2$, then the Calabi--Yau $(k-1)$-fold $X_{(2,2,\ldots,2)} \subset \PP^{3k-1}$ admits a CY 0-category fibration over $\PP^{k-1}$, which indicates that they should always admits a ramified finite cover map to $\PP^{k-1}$.

	\item If we take $d=3$, then the Calabi--Yau $(2k-1)$-fold $X_{(3,3,\ldots, 3)} \subset \PP^{3k-1}$ admits a Calabi--Yau $k$-category fibration over $\PP^{k-1}$. 
\end{itemize}

\subsection{Other examples}
Notice that even in the well know examples, the statement of the theorem in the the critical case $\ell = m$ implies something nontrivial.

\begin{example}[Pfaffian-Grassmannian correspondence] \label{eg:Pf-Gr} Let $X = \Gr(2,7) \subset \PP^{20}$ through the Pl\"ukcer embedding, then it has a natural rectangular Lefschetz structure of length $m=7$, and its HPD is given by the noncommutative resolution of Pfaffian loci $Y = (\Pf(4,7), \shR) \subset \check{\PP}^{20}$, see \cite{Kuz06, BC06}. Let $L \subset (\CC^{20})^\vee$ be a generic linear system of dimension $7$, then $X_{L^\perp}$ and $Y_L$ are non-birational Calabi-Yau threefolds. The theorem implies the refined blowing up $\Bl^{\rm ref}_{X_{L^\perp}} X$ of $X$ along $X_{L^\perp}$ has a Lefschetz decomposition with respect to $\sO_{\PP(L^\vee)}(1)$:
	$$D(\Bl^{\rm ref}_{X_{L^\perp}} X) = \langle D(X_{L^\perp}), D(X_{L^\perp})(1), \ldots, D(X_{L^\perp}) (5)\rangle,$$
and is HPD to $Y_L$ with the trivial decomposition $D(Y_L) = D(Y_L)$. This result combined with (1) of Thm. \ref{thm:HPD}  gives another proof of the derived equivalence
	$$D(X_{L^\perp}) \simeq D(Y_L)$$
of \cite{Kuz06, BC06}. Note that the theorem also implies $\Bl^{\rm ref}_{X_{L^\perp}} X$ is equivalent  to the universal hyperplane section $\shH_{Y_L}$ for $Y_L$, which is not obvious from geometry.
\end{example}

\begin{example}[Beauville-Donagi's Pfaffian cubic and K3 surface] \label{eg:Pf-K3} Let $X = \Gr(2,6) \subset \PP^{14}$ with Pl\"ucker embedding and Lefschetz decomposition given in \cite{Kuz06}, then its HPD is given by  $Y = (\Pf(4,6), \shR) \subset \check{\PP}^{14}$. Let $L = \CC^{6} \subset (\CC^{15})^\vee$ be a generic linear subspace of dimension $\ell = 6$. Then $X_{L^\perp} = S \subset \Gr(2,6)$ is a K3 surface, and $Y_L := Y_{4}(3) \subset \PP^5$ is a cubic fourfold. The theorem implies that the refined blowing up of $\Gr(2,6)$ along the K3 surface $S$:
	$$D(\Bl^{\rm ref}_{S} \Gr(2,6)) = \langle D(S), \sO, D(S)(1), \sO(1), D(S)(2), \sO(2), D(S)(3), D(S)(4) \rangle,$$
is HPD to the Pfaffian cubic fourfold $Y_{4}(3) \subset \PP^{5}$ with decomposition
	$$D(Y_{4}(3)) = \langle \shC_L, \sO, \sO(1), \sO(2) \rangle.$$
From (1) of Thm. \ref{thm:HPD}, this implies the well-known result $\shC_L = D(S)$, i.e. there is a geometric K3 surface $S$ associated to the Pfaffian cubic fourfold $Y_{4}(3) \subset \PP^5$. For a general cubic fourfold it is expected that $\shC_L$ is only a noncommutative K3 surface, and whether it is geometric or not is closely related to the rationality of $Y_{4}(3)$. See \cite{Kuz10K3, AT14, Huy17} and references therein for more details. 
\end{example}


\end{document}